\documentclass[11pt]{amsart}
\usepackage{latexsym,amscd,amssymb, graphicx, color, amsthm, bm, amsmath, tikz}  
\usepackage[margin=1in]{geometry}
\usepackage{hyperref}
\usepackage[all,cmtip]{xy}

\numberwithin{equation}{section}

\newtheorem{theorem}{Theorem}[section]
\newtheorem{proposition}[theorem]{Proposition}
\newtheorem{corollary}[theorem]{Corollary}
\newtheorem{lemma}[theorem]{Lemma}
\newtheorem{conjecture}[theorem]{Conjecture}
\newtheorem{observation}[theorem]{Observation}

\newtheorem{problem}[theorem]{Problem}
\newtheorem{example}[theorem]{Example}

\theoremstyle{definition}
\newtheorem{defn}[theorem]{Definition}

\newcommand{\rvline}{\hspace*{-\arraycolsep}\vline\hspace*{-\arraycolsep}}

\newcommand{\PM}{{\mathrm {PM}}}
\newcommand{\Rise}{{\mathrm {Rise}}}
\newcommand{\Stir}{{\mathrm {Stir}}}

\newcommand{\grFrob}{{\mathrm {grFrob}}}

\newcommand{\rev}{{\mathrm {rev}}}

\newcommand{\code}{{\tt {code}}}

\newcommand{\Frob}{{\mathrm {Frob}}}

\newcommand{\initial}{{\mathrm {in}}}

\newcommand{\rk}{{\mathrm {rk}}}

\newcommand{\neglex}{{\mathtt {neglex}}}

\newcommand{\symm}{{\mathfrak{S}}}

\newcommand{\CC}{{\mathbb {C}}}
\newcommand{\QQ}{{\mathbb {Q}}}
\newcommand{\ZZ}{{\mathbb {Z}}}

\newcommand{\OP}{{\mathcal{OP}}}
\newcommand{\FF}{{\mathbb{F}}}

\newcommand{\JJJ}{{\mathcal{J}}}
\newcommand{\UUU}{{\mathcal{U}}}
\newcommand{\CCC}{{\mathcal{C}}}
\newcommand{\VVV}{{\mathcal {V}}}

\newcommand{\MMM}{{\mathcal{M}}}
\newcommand{\BBB}{{\mathcal{B}}}
\newcommand{\III}{{\mathcal{I}}}

\newcommand{\Val}{{\mathrm{Val}}}

\newcommand{\zz}{{\mathbf {z}}}
\newcommand{\xx}{{\mathbf {x}}}
\newcommand{\II}{{\mathbf {I}}}

\newcommand{\TT}{{\mathbf {T}}}

\begin{document}

\title[Spanning subspace configurations]
{Spanning subspace configurations}


\author{Brendon Rhoades}
\address
{Department of Mathematics \newline \indent
University of California, San Diego \newline \indent
La Jolla, CA, 92093-0112, USA}
\email{bprhoades@math.ucsd.edu}

\begin{abstract}
A {\em spanning configuration} in the complex vector space $\CC^k$ is a 
sequence $(W_1, \dots, W_r)$ of linear subspaces of $\CC^k$ such that $W_1 + \cdots + W_r = \CC^k$.
We present the integral cohomology of the moduli space of spanning configurations in 
$\CC^k$ corresponding
to a given sequence of subspace dimensions.
This simultaneously
generalizes the classical presentation of the cohomology of partial flag varieties and 
the more recent presentation of a variety of spanning line configurations defined by the author and Pawlowski.
This latter variety of spanning line configurations plays the role of the flag variety for the Haglund-Remmel-Wilson
Delta Conjecture of symmetric
function theory.
\end{abstract}

\keywords{subspace configuration, cohomology}
\maketitle

\begin{figure}[h]
\begin{tikzpicture}[scale = 0.7]

\draw [gray, fill=red!20] (10.5,0.5) -- (14.5,0.5) -- (13.5,-0.5) -- (9.5,-0.5) -- (10.5,0.5);

\draw [gray, fill=red!20] (12.5,2.5) -- (11.5,1.5) -- (11.5,-2.5) -- (12.5,-1.5) -- (12.5,2.5);


\draw [gray] (12.5, -0.5) -- (11.5,-0.5);

\draw [red!20] (12.5, -0.5) -- (12.5,0);

\draw [red!20] (12.5, 0) -- (12.5,0.5);

\draw  (12,-2) -- (12,2);

\draw (10,0) -- (14,0);

\draw  (12.5,0.5) -- (11.5,-0.5);

\draw [line width = 0.5mm, blue] (13,2) -- (11,-2);

\node at (14.3,0.8) {$W_3$};

\node at (13.4,1.8) {$W_2$};

\node at (11.6,2.1) {$W_1$};


\end{tikzpicture} 
\caption{A point in $X_{(2,1,2),3}$.}
\label{fig:point}
\end{figure}
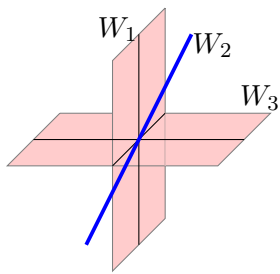

\section{Introduction}
\label{Introduction}

In this paper we compute the cohomology of the moduli space of ordered tuples of subspaces 
of $\CC^k$ with a given sequence of dimensions which have full span $\CC^k$.
These varieties generalize classical spaces of Schubert calculus 
(partial flag varieties) as well as 
a newer space $X_{n,k}$ of spanning line configurations which has connections to the
Delta Conjecture of symmetric function theory \cite{PR}.
The cohomology rings of these subspace configurations have combinatorics
governed by ordered tuples of subsets $\III = (I_1, \dots, I_r)$ of
$[k] := \{1, 2, \dots, k\}$ which satisfy $I_1 \cup \cdots \cup I_r = [k]$.

A finite sequence $(W_1, \dots, W_r)$ of linear subspaces of the $k$-dimensional complex vector 
space
$\CC^k$ is a {\em spanning configuration} 
if the $W_i$ span $\CC^k$, i.e. $W_1 + \cdots + W_r = \CC^k$.
The {\em dimension vector} of 
$(W_1, \dots, W_r)$ of subspaces of $\CC^k$ is the sequence of dimensions
$(\dim(W_1), \dots, \dim(W_r))$.
The following moduli space of spanning configurations is our object of study.

\begin{defn}
\label{main-space-definition}
Let $k > 0$ and let $\alpha = (\alpha_1, \dots, \alpha_r) \in [k]^r$. 
Let $X_{\alpha,k}$ be the moduli space
\begin{equation}
X_{\alpha,k} := \{ \text{all spanning configurations $(W_1, \dots, W_r)$ in $\CC^k$ with dimension vector 
$\alpha$} \}.
\end{equation}
\end{defn}

A subspace configuration in $X_{(2,1,2),3}$ is shown in Figure~\ref{fig:point}.

Let $Gr(d,k)$ be the Grassmannian of all $d$-dimensional subspaces of $\CC^k$. We have an inclusion
\begin{equation}
\iota: X_{\alpha,k} \hookrightarrow Gr(\alpha,k)
\end{equation}
of $X_{\alpha,k}$ into the product of Grassmannians
\begin{equation}
Gr(\alpha,k) := Gr(\alpha_1,k) \times \cdots \times Gr(\alpha_r,k). 
\end{equation}
This realizes $X_{\alpha,k}$ as a Zariski open subset of $Gr(\alpha,k)$.
We may therefore view $X_{\alpha,k}$ as either a complex algebraic variety or a smooth complex manifold.

Various old and new moduli spaces arising in Schubert calculus are special cases of the $X_{\alpha,k}$.
\begin{itemize}
\item  If $\alpha_i = k$ for any $i$, then any sequence of subspaces in $\CC^k$ is a spanning configuration
and $X_{\alpha,k}$ equals the Grassmann product $Gr(\alpha,k)$.
Any product of Grassmannians can be obtained in this way.
\item  If $\alpha = (1^k)$ is a length $k$ sequence of $1$'s, then $X_{(1^k),k}$ is homotopy equivalent 
to the flag variety 
\begin{equation}
\mathcal{F \ell}(k) = \{ 0 = V_0 \subset V_1 \subset \cdots \subset V_k = \CC^k \,:\, \dim(V_i) = i \}
\end{equation}
of complete flags in $\CC^k$.  To see why this is the case, let 
$GL_k$ be the group of $k \times k$ invertible complex matrices and let $T \subseteq B$ be the subgroups
of diagonal and upper triangular matrices.  We have a canonical projection
\begin{equation}
\pi: X_{(1^k),k} = GL_k/T \twoheadrightarrow GL_k/B = \mathcal{F \ell}(k).
\end{equation}
It can be shown that $\pi$ is a fiber bundle with contractible fiber (homeomorphic to the maximal unipotent subgroup
of $GL_k$ of upper triangular matrices with $1$'s on the diagonal), so that $\pi$ is a homotopy equivalence.
\item  More generally, if $\alpha_1 + \cdots + \alpha_r = k$, then $X_{\alpha,k}$ is homotopy equivalent to
the partial flag variety
\begin{equation}
\mathcal{F \ell}(\alpha) = \{ 0 = V_0 \subset V_1 \subset \cdots \subset V_r = \CC^k \,:\,
\dim(V_i) = \alpha_1 + \cdots + \alpha_i \}.
\end{equation}
Any partial flag variety can be obtained in this way.
\item  Finally, if $\alpha = (1^n)$ for some $k \leq n$, the space $X_{(1^n),k}$ 
consists of $n$-tuples of lines whose span equals $\CC^k$. This space 
of {\em spanning line configurations}
was introduced by 
Pawlowski and the author  \cite{PR} under the name $X_{n,k}$. It was shown that the cohomology 
$H^{\bullet}(X_{(1^n),k}; \ZZ)$ has rank equal to $k! \cdot \Stir(n,k)$, where $\Stir(n,k)$ is the Stirling number
of the second kind counting set partitions of an $n$-element set into $k$ nonempty blocks.
\end{itemize}

The space $X_{(1^n),k}$ of the last bullet point  has connections to symmetric function theory.
We briefly describe this connection here
and refer the reader to \cite{Bergeron} for details on symmetric function theory.

Given $k \leq n$, the {\em Delta Conjecture} of 
Haglund, Remmel, and Wilson \cite{HRW} predicts the following identity
of formal power series involving infinitely many variables $\xx = (x_1, x_2, \dots )$ and two additional
parameters $q,t$:
\begin{equation}
\Delta'_{e_{k-1}} e_n = \Rise_{n,k}(\xx;q,t) = \Val_{n,k}(\xx;q,t).
\end{equation}
Here $e_n$ is the elementary symmetric function, $\Delta'_{e_{k-1}}$ is the  primed
delta operator attached to $e_{k-1}$, and $\Rise$ and $\Val$ are certain combinatorially defined formal
power series related to lattice paths.

The full Delta Conjecture is open, but it is known when one of the variables $q,t$ is specialized
to zero.
Combining results in \cite{GHRY, HRW, HRS, Rhoades, Wilson} we have
\begin{equation}
\label{six-equality}
\Delta'_{e_{k-1}} e_n \mid_{t = 0} = 
\Rise_{n,k}(\xx;q,0) = \Rise_{n,k}(\xx;0,q) = \Val_{n,k}(\xx;q,0) = \Val_{n,k}(\xx;0,q).
\end{equation}
Let $C_{n,k}(\xx;q)$ be the common symmetric function of Equation~\eqref{six-equality}.

Haglund, the author, and Shimozono \cite{HRS} gave an algebraic model 
and Pawlowski and the author \cite{PR} gave a geometric model
for the symmetric function $C_{n,k}(\xx;q)$.
Let $\ZZ[\xx_n] := \ZZ[x_1, \dots, x_n]$ and $\QQ[\xx_n] := \QQ[x_1, \dots, x_n]$ be the polynomial
rings in $n$ variables over the ring of integers and field of rational numbers.
Recall the {\em elementary symmetric polynomial} $e_d(\xx_n)$
and {\em complete homogeneous symmetric polynomial} $h_d(\xx_n)$
of degree $d$ in the variable set $\xx_n$:
\begin{equation}
e_d(\xx_n) := \sum_{1 \leq i_1 < \cdots < i_d \leq n} x_{i_1} \cdots x_{i_d}, \quad
h_d(\xx_n) := \sum_{1 \leq i_1 \leq \cdots \leq i_d \leq n} x_{i_1} \cdots x_{i_d}.
\end{equation}

For $k \leq n$, we follow \cite{HRS, PR} and define the ideals
\begin{align}
I_{n,k} &:= \langle x_1^k, x_2^k, \dots, x_n^k, e_n(\xx_n), e_{n-1}(\xx_n), \dots, e_{n-k+1}(\xx_n) \rangle
\subseteq \ZZ[\xx_n], \\
J_{n,k} &:= \langle x_1^k, x_2^k, \dots, x_n^k, e_n(\xx_n), e_{n-1}(\xx_n), \dots, e_{n-k+1}(\xx_n) \rangle
\subseteq \QQ[\xx_n].
\end{align}
Let $R_{n,k} := \ZZ[\xx_n]/I_{n,k}$ and $S_{n,k} := \QQ[\xx_n]/J_{n,k}$ be the corresponding quotient rings.

Let $\symm_n$ denote the symmetric group on $n$ letters. 
We recall the Frobenius image encoding of the isomorphism type of a (graded) $\symm_n$-module 
as a symmetric function.

The irreducible representations
of $\symm_n$ are in one-to-one correspondence with partitions $\lambda \vdash n$. 
If $\lambda \vdash n$ is a partition, let $S^{\lambda}$ denote the corresponding irreducible 
$\symm_n$-module. Any finite-dimensional $\symm_n$-module $V$ may be expressed as
$V \cong \bigoplus_{\lambda \vdash n} m_{\lambda} S^{\lambda}$ for some unique multiplicities
$m_{\lambda} \geq 0$.  The {\em Frobenius image} of $V$ is the symmetric function
$\Frob(V) := \sum_{\lambda \vdash n} m_{\lambda} s_{\lambda}$, where 
$s_{\lambda}$ is the Schur function.
Furthermore, if $V = \bigoplus_{d \geq 0} V_d$ is a graded $\symm_n$-module with each
piece $V_d$ finite-dimensional and $q$ is a grading variable, the {\em graded Frobenius image} is
$\grFrob(V;q) := \sum_{d \geq 0} \Frob(V_d) \cdot q^d$.

\begin{theorem}
\label{line-configuration-theorem}
Let $k \leq n$ be positive integers.
\begin{enumerate}
\item  (Pawlowski-R. \cite{PR})  The singular cohomology ring $H^{\bullet}(X_{(1^n),k}; \ZZ)$ may be presented
as 
\begin{equation*}
H^{\bullet}(X_{(1^n),k}; \ZZ) = R_{n,k},
\end{equation*}
where the variable $x_i$ represents the Chern class $c_1(\ell_i^*) \in H^2(X_{(1^n)}; \ZZ)$.
\item  (Haglund-R.-Shimozono \cite{HRS})  The graded Frobenius image of the graded $\symm_n$-modue
 $S_{n,k}$ is given by
 \begin{equation*}
 \grFrob(S_{n,k}; q) = (\rev_q \circ \omega) C_{n,k}(\xx;q).
 \end{equation*}
 Here $\rev_q$ is the operator which reverses the coefficient sequences of polynomials in $q$
 and $\omega$ is the symmetric function involution which trades $e_n$ and $h_n$. 
\end{enumerate}
\end{theorem}

Since the cohomology of $X_{(1^n),k}$ is concentrated in even dimensions, the Universal Coefficient
Theorem and the above result imply that
\begin{equation}
 \grFrob(H^{\bullet}(X_{(1^n),k}; \QQ); q) = (\rev_q \circ \omega) C_{n,k}(\xx;q).
\end{equation}
In this sense, the space $X_{(1^n),k}$ furnishes a geometric model for the Delta Conjecture.

Now consider the spaces $X_{\alpha,k}$ in the case $k = n$.  That is, suppose we have a sequence
$\alpha = (\alpha_1, \dots, \alpha_r) \in [n]^r$ such that $\alpha_1 + \cdots + \alpha_r = n$.
In this case, the space $X_{\alpha,n}$ is homotopy equivalent to the partial flag variety
$\mathcal{F \ell}(\alpha)$.
There is a classical presentation of the cohomology ring $H^{\bullet}(\mathcal{F \ell}(\alpha); \ZZ)$ due 
to Borel. To state it, we introduce some notation.

Let $\alpha = (\alpha_1, \dots, \alpha_r)$ be a length $r$ sequence 
of positive integers with $\alpha_1 + \cdots + \alpha_r = n$.  
We break up the list $\xx_n = (x_1, \dots, x_n)$ of variables into $r$ sublists 
$\xx^{(1)}_n, \dots, \xx^{(r)}_n$ where 
\begin{equation}
\xx^{(i)}_n := (x_{\alpha_1 + \cdots + \alpha_{i-1} + 1}, \dots , x_{\alpha_1 + \cdots + \alpha_{i-1} + \alpha_i})
\end{equation}
for $1 \leq i \leq r$. For example, if $\alpha = (2,1,2)$ so that $r = 3$ and $n = 5$ then 
\begin{equation*}
\xx^{(1)}_5 = (x_1, x_2), \quad \xx^{(2)}_5 = (x_3), \quad \xx^{(3)}_5 = (x_4, x_5).
\end{equation*}
Let $\symm_{\alpha} := \symm_{\alpha_1} \times \cdots \times \symm_{\alpha_r}$ be the parabolic subgroup
of $\symm_n$ attached to $\alpha$.
The factor $\symm_{\alpha_i}$ of $\symm_{\alpha}$ acts on the subscripts of variables in 
$\xx^{(i)}_n$.
Let $\ZZ[\xx_n]^{\symm_{\alpha}}$ be the ring of polynomials in $\xx_n$ which are invariant
under the action of $\symm_{\alpha}$.

\begin{theorem} 
\label{borel-theorem}
(Borel \cite{Borel})
Let $\alpha = (\alpha_1, \dots, \alpha_r) \in [k]^r$ be a sequence of positive integers 
with $\alpha_1 + \cdots + \alpha_r = n$.
The cohomology ring $H^{\bullet}(\mathcal{F \ell}(\alpha); \ZZ) = H^{\bullet}(X_{\alpha,k}; \ZZ)$ 
may be presented as 
\begin{equation}
H^{\bullet}(\mathcal{F \ell}(\alpha); \ZZ) = (\ZZ[\xx_n]/I)^{\symm_{\alpha}},
\end{equation}
where $I \subseteq \ZZ[\xx_n]$ is the ideal generated by the complete homogeneous
symmetric polynomials $h_d(\xx_n^{(i)})$ where $d > k - \alpha_i$ and $1 \leq i \leq r$.
\end{theorem}

Our main result presents the cohomology ring $H^{\bullet}(X_{\alpha,k}; \ZZ)$ 
of $X_{\alpha,k}$ for an arbitrary dimension vector $\alpha \in [k]^r$.
This simultaneously generalizes Borel's presentation of the cohomology of partial flag varieties
in Theorem~\ref{borel-theorem} and the 
geometric model for the Delta Conjecture due to Pawlowski and the author in 
Theorem~\ref{line-configuration-theorem}.

\begin{theorem}
\label{main-theorem}
Let $\alpha = (\alpha_1, \dots, \alpha_r) \in [k]^r$ be a sequence of positive integers with 
$\alpha_1 + \cdots + \alpha_r = n$.  
The singular cohomology ring $H^{\bullet}(X_{\alpha,k}; \ZZ)$ may be presented as 
\begin{equation}
H^{\bullet}(X_{\alpha,k}; \ZZ) = 
(\ZZ[\xx_n]/I_{\alpha,k})^{\symm_{\alpha}},
\end{equation}
where $I_{\alpha,k} \subseteq \ZZ[\xx_n]$ is the ideal generated by 
the complete homogeneous
symmetric polynomials $h_d(\xx_n^{(i)})$ where $d > k - \alpha_i$ and $1 \leq i \leq r$
together with the elementary symmetric polynomials $e_n(\xx_n), e_{n-1}(\xx_n), \dots, e_{n-k+1}(\xx_n)$.
Here the variables  $\xx^{(i)}_n$ represent the Chern roots of the tautological vector bundle 
$W_i^* \twoheadrightarrow X_{\alpha,k}$.
\end{theorem}

Theorem~\ref{main-theorem} will be proven in Section~\ref{Cohomology}.
The ring $H^{\bullet}(X_{\alpha,k};\ZZ)$ is a free $\ZZ$-module and we can describe its rank 
combinatorially.  Given $\alpha = (\alpha_1, \dots, \alpha_r) \in [k]^r$, the rank of 
$H^{\bullet}(X_{\alpha,k}; \ZZ)$ equals the number of $0,1$-matrices $A$ of size $k \times r$
with column sum vector $\alpha$ containing no row of zeros.
For example, if $k = 3$ and $\alpha = (2,1,2,1)$ a matrix contributing to the rank of 
$H^{\bullet}(X_{(2,1,2,1)}; \ZZ)$ is 
\begin{equation*}
\begin{pmatrix}
0 & 1 & 1 & 0 \cr
1 & 0 & 1 & 1 \cr
1 & 0 & 0 & 0 
\end{pmatrix}.
\end{equation*}

When $\alpha = (d, \dots, d) \in [k]^r$ is a constant vector, the symmetric group $\symm_r$
acts on $X_{\alpha,k}$ by
subspace permutation:
\begin{equation}
\pi.(W_1, \dots, W_r) := (W_{\pi(1)}, \dots, W_{\pi(r)} )
\end{equation}
for $\pi \in \symm_r$ and $(W_1, \dots, W_r) \in X_{\alpha,k}$.
This induces an action on the rational cohomology $H^{\bullet}(X_{\alpha,k}; \QQ)$
which coincides with the action of $\symm_r$ on the columns of the 
0,1-matrices given in the last paragraph.

The proof of Theorem~\ref{main-theorem} is a mixture of classical and modern 
techniques.
We describe the cohomology of $H^{\bullet}(X_{\alpha,k}; \ZZ)$ as a graded group by 
describing a nonstandard affine paving of the Grassmann product $Gr(\alpha,k)$ which 
interacts nicely with the inclusion $X_{\alpha,k} \subseteq Gr(\alpha,k)$.
This affine paving relies on a linear algebra algorithm combining
row and column operations which we call `mixed reduction'
(see Section~\ref{Mixed}).
The analysis of the quotient ring $(\ZZ[\xx_n]/I_{\alpha,k})^{\symm_{\alpha}}$ relies
on the method of {\em orbit harmonics} and, in particular, a Demazure character identity
which relies on a refinement of the dual Pieri rule due to Haglund, Luoto, Mason,
and van Willigenburg \cite{HLMV}.

The remainder of the paper is organized as follows.
In {\bf Section~\ref{Background}} we give background on the combinatorial objects we will be using,
Gr\"obner theory, affine pavings of varieties, and Chern classes.
In {\bf Section~\ref{Quotient}} we will use the method of orbit harmonics to analyze 
the quotient $\ZZ[\xx_n]/I_{\alpha,k}$ (upon extending coefficients to $\QQ$).
In {\bf Section~\ref{Mixed}} we introduce our linear algebra algorithm of mixed reduction
and apply it to get a strategic affine paving of $Gr(\alpha,k)$ which restricts to an affine
paving of $X_{\alpha,k}$.
In {\bf Section~\ref{Cohomology}} we prove Theorem~\ref{main-theorem} and present the cohomology 
of $X_{\alpha,k}$.
We close in {\bf Section~\ref{Open}} with some open problems.

\section{Background}
\label{Background}

\subsection{Combinatorics}
A {\em (strong) composition} is a finite sequence $\alpha = (\alpha_1, \dots, \alpha_r)$ of positive integers.
If $\alpha_1 + \cdots + \alpha_r = n$, we say that $\alpha$ is a {\em composition of $n$} and write
 $|\alpha| = n$.  We also say that $\alpha$ has {\em $r$ parts} and write $\ell(\alpha) = r$.
We will  be concerned with compositions $\alpha \in [k]^r$ whose parts are bounded above
by a constant $k$.

Let $\alpha = (\alpha_1, \dots, \alpha_r) \in [k]^r$ with $|\alpha| = n$ and suppose we have an ordered list 
of $n$ things (such as variables in $(x_1, x_2, \dots, x_n)$ or the columns
of an $n$-column matrix read from left to right).
For $1 \leq i \leq r$, we will refer to the subset of these things indexed by the set
\begin{equation*}
\{ \alpha_1 + \cdots + \alpha_{i-1} + 1, 
\alpha_1 + \cdots + \alpha_{i-1} + 2, \dots, \alpha_1 + \cdots + \alpha_{i-1} + \alpha_i\}
\end{equation*} 
as the 
$i^{th}$ {\em batch} (with dependence on $\alpha$ understood).
For example, if $\alpha = (2,1,2)$ and we have the list of variables $\xx_5 = (x_1, x_2, x_3, x_4, x_5)$, the first 
batch is $\xx_5^{(1)} = (x_1, x_2)$, the second batch is $\xx_5^{(2)} = (x_3)$, and the third batch is 
$\xx_5^{(3)} = (x_4, x_5)$.

If  $\alpha = (\alpha_1, \dots, \alpha_r) \in [k]^r$ with $|\alpha| = n$, let
$\symm_{\alpha} := \symm_{\alpha_1} \times  \cdots \symm_{\alpha_r} \subseteq \symm_n$
be the associated
{\em parabolic subgroup} of $\symm_n$. The factor $\symm_{\alpha_i}$ permutes the 
$i^{th}$ batch of letters in $[n]$.
We also let $L_{\alpha}$ be the {\em Levi subgroup} of $GL_n(\CC)$ attached to $\alpha$.
Elements of $L_{\alpha}$ are block diagonal matrices of the form $A_1 \oplus \cdots \oplus A_r$ where
$A_i \in GL_{\alpha_i}$.
For example, matrices in $L_{(2,1,2)}$ have the form
\begin{scriptsize}
\begin{equation*}
\begin{pmatrix}
 \star & \star & 0 & 0 & 0 \\
 \star & \star & 0 &0 & 0 \\
 0 & 0 & \star & 0 & 0 \\
 0 & 0 & 0 & \star & \star \\
 0 & 0 & 0 & \star & \star
\end{pmatrix}
\end{equation*}
\end{scriptsize}
where the $\star$'s are complex numbers and three diagonal blocks are invertible.

An {\em ordered set partition} of $[n]$ is a sequence $\sigma = (B_1 \mid \dots \mid B_k)$ 
of nonempty 
subsets of $[n]$ such that $[n] = B_1 \sqcup \cdots \sqcup B_k$ (disjoint union).
We say that $\sigma$ has {\em size $n$} and {\em $k$ blocks}.
Let $\OP_{n,k}$ be the family of ordered set partitions of $[n]$ with $k$ blocks.
For example, we have $(2 \, 5  \mid 3  \mid 1 \, 4 ) \in \OP_{5,3}$.

Suppose $\alpha \in [k]^r$ with $|\alpha| = n$. 
We consider the following collection
$\OP_{\alpha,k}$ of ordered set partitions:
\begin{equation}
\OP_{\alpha,k} := \left\{ \sigma \in \OP_{n,k} \,:\,
\begin{array}{c}
\text{for each $1 \leq i \leq \ell(\alpha)$, the $i^{th}$ batch of letters} \\
\text{$\alpha_1 + \cdots + \alpha_{i-1} + 1, \alpha_1 + \cdots + \alpha_{i-1} + 2, \dots, 
\alpha_1 + \cdots + \alpha_i$} \\ \text{are contained in distinct blocks of $\sigma$}
\end{array} \right \}.
\end{equation}
If $\alpha = (2,1,2)$ and $k = 3$ then 
$(2 \, 5  \mid 3  \mid 1 \, 4 ) \in \OP_{\alpha,k}$ but 
$( 4 \, 5  \mid 2   \mid 1 \, 3) \notin \OP_{\alpha,k}$ (because $4$ and $5$ lie in the same block).

For any $\alpha \in [k]^r$ with $|\alpha| = n$, we have $\OP_{\alpha,k} \subseteq \OP_{n,k}$.
The set $\OP_{n,k}$ carries an action of $\symm_n$ by letter permutation.
The subset $\OP_{\alpha,k}$ is 
typically not closed under the action of the fully symmetric group $\symm_n$, but is
stable under the action of the parabolic subgroup $\symm_{\alpha}$.

Let $k \geq 0$.
A {\em set sequence in $[k]$} is a finite sequence $\III = (I_1, \dots, I_r)$ of 
nonempty subsets of $[k]$. 
The {\em type} of $\III$ is the composition $(|I_1|, \dots, |I_r|) \in [k]^r$ obtained by taking the cardinalities 
of the sets.
The sequence $\III$ is said to {\em cover $[k]$} if $I_1 \cup \cdots \cup I_r = [k]$.
We omit set braces when writing set sequences, so that when $k = 3$ the sequence
$(13 , 3 , 23, 1 )$ is a set sequence of type $(2,1,2,1)$ which covers $[3]$
whereas $(23, 2, 23,3)$ is a set sequence of type $(2,1,2,1)$ which does not cover $[3]$. 
We have a natural correspondence
\begin{equation*}
\text{$\symm_{\alpha}$-orbits in $\OP_{\alpha,k}$} \leftrightarrow
\text{set sequences of type $\alpha$ which cover $[k]$}.
\end{equation*}

We can view set sequences as 0,1-matrices, so that the examples of the last paragraph become
\begin{equation*}
(13 , 3 , 23, 1) \leftrightarrow \begin{pmatrix}
1 & 0  & 0 & 1 \\
0 & 0 & 1 & 0 \\
1 & 1 & 1 & 0 
\end{pmatrix} \quad \text{ and } \quad
(23, 2, 23, 3) \leftrightarrow \begin{pmatrix}
0 & 0 & 0 & 0 \\
1 & 1 & 1 & 0 \\
1 & 0 & 1 & 1
\end{pmatrix}.
\end{equation*}
The type of the set sequence is the column sum vector of the matrix.
Covering set sequences  correspond to matrices without zero rows.

\subsection{Gr\"obner theory}
A total order $<$ on the monomials in $\QQ[\xx_n]$ is a {\em monomial order} if
\begin{enumerate}
\item $1 \leq m$ for any monomial $m$ and
\item if $m_1, m_2, m_3$ are monomials and $m_1 < m_2$, then $m_1 \cdot m_3 < m_2 \cdot m_3$.
\end{enumerate}
In this paper, we shall only consider the monomial order defined by 
$x_1^{a_1} \cdots x_n^{a_n} < x_1^{b_1} \cdots x_n^{b_n}$ if there exists $1 \leq i \leq n$ such that 
$a_i < b_i$ and $a_{i+1} = b_{i+1}, \dots, a_n = b_n$.  
This is the lexicographical order with respect to the reversed variable set $x_n > \cdots > x_1$.
This `negative lexicographical' order is called $\neglex$.

If $<$ is a monomial order and $f \in \QQ[\xx_n]$ is a nonzero polynomial, let $\initial_<(f)$ be the leading 
monomial of $f$ with respect to $<$.  If $I \subseteq \QQ[\xx_n]$ is an ideal, the {\em initial ideal}
 $\initial_<(I) \subseteq \QQ[\xx_n]$ is the monomial ideal 
 \begin{equation}
 \initial_<(I) := \langle \initial_<(f) \,:\, f \in I - \{0\} \rangle.
 \end{equation}
 
Let $I \subseteq \QQ[\xx_n]$ be an ideal.
 A finite subset $G = \{g_1, \dots, g_t\} \subseteq I$ is a {\em Gr\"obner basis} of $I$ if 
 $\initial_<(I) = \{ \initial_<(g_1), \dots, \initial_<(g_t) \}$. Equivalently, the leading term of any nonzero polynomial
 in $I$ is divisible by the leading term of one of the polynomials in $G$.  If $G$ is a Gr\"obner basis for $I$,
 we have $I = \langle G \rangle$ as ideals.

 \subsection{Demazure characters}
 For $1 \leq i \leq n-1$, the {\em divided difference operator} $\partial_i: \ZZ[\xx_n] \rightarrow \ZZ[\xx_n]$
 is defined by
 \begin{equation}
 \partial_i(f(\xx_n)) := \frac{f(x_1, \dots, x_i, x_{i+1}, \dots, x_n) - f(x_1, \dots, x_{i+1}, x_i, \dots, x_n)}{x_i - x_{i+1}}.
 \end{equation}
 The {\em isobaric divided difference operator} $\pi_i: \ZZ[\xx_n] \rightarrow \ZZ[\xx_n]$ is given by
 $\pi_i(f(\xx_n)) := \partial_i(x_i f(\xx_n))$.
 
 A {\em weak composition} $\gamma$ of length $n$ is a sequence $\gamma = (\gamma_1, \dots, \gamma_n)$
 of $n$ nonnegative integers. The {\em Demazure character} 
 (or {\em key polynomial}) $\kappa_{\gamma}(\xx_n) \in \ZZ[\xx_n]$ 
 is a polynomial defined recursively by the following two rules.
 \begin{enumerate}
 \item
 If $\gamma_1 \geq \cdots \geq \gamma_n$ then $\kappa_{\gamma}(\xx_n)$ is the monomial
 $x_1^{\gamma_1} \cdots x_n^{\gamma_n}$.  
 \item If $\gamma_i < \gamma_{i+1}$ then
 \begin{equation}
 \kappa_{(\gamma_1, \dots, \gamma_{i+1}, \gamma_i, \dots, \gamma_n)}(\xx_n) =
 \pi_i( \kappa_{(\gamma_1, \dots, \gamma_{i}, \gamma_{i+1}, \dots, \gamma_n)}(\xx_n)).
 \end{equation}
 \end{enumerate}
Since the isobaric divided difference operators satisfy the 0-Hecke relations
 \begin{equation}
 \begin{cases}
 \pi_i^2 = \pi_i & \text{for $1 \leq i \leq n-1$,} \\
 \pi_i \pi_j = \pi_j \pi_i & \text{for $|i-j| > 1$,}  \\
 \pi_i \pi_{i+1} \pi_i = \pi_{i+1} \pi_i \pi_{i+1} & \text{for $1 \leq i \leq n-2$,}
 \end{cases}
 \end{equation}
 these rules give a well-defined polynomial $\kappa_{\gamma}(\xx_n)$ for any weal composition $\gamma$
 of length $n$.
We will need the following 
 two properties of $\kappa_{\gamma}(\xx_n)$:
 \begin{enumerate}
 \item the leading monomial of $\kappa_{\gamma}(\xx_n)$ with respect to the $\neglex$ order 
 is $x_1^{\gamma_1} \cdots x_n^{\gamma_n}$, and
 \item the coefficients of $\kappa_{\gamma}(\xx_n)$ are  integers and the coefficient of the monomial
 $x_1^{\gamma_1} \cdots x_n^{\gamma_n}$ is 1.
 \end{enumerate}

\subsection{Cohomology}
We refer the reader to \cite{Fulton} for a thorough reference on the geometric objects used
in this paper.

An {\em affine paving} of a complex variety $X$ 
is a sequence
\begin{equation}
\varnothing = X_0 \subset X_1 \subset \cdots \subset X_m = X
\end{equation}
of closed subvarieties of $X$ such that for $1 \leq i \leq m$, the difference $X_i - X_{i-1}$ is isomorphic to
a disjoint union $\bigsqcup_j A_{ij}$ of complex affine spaces $A_{ij}$.
The affine spaces $A_{ij}$ are called the {\em cells} of the affine paving and we say that 
the set $\{ A_{ij} \}$ of all cells {\em induces an affine paving} of $X$.

Let $X$ be a smooth irreducible complex algebraic variety.  Assume that $X$ 
admits an affine paving 
$\varnothing = X_0 \subset X_1 \subset \cdots \subset X_m = X$.  We will need the following two geometric
facts.
\begin{enumerate}
\item  The (singular) cohomology $H^{\bullet}(X; \ZZ)$ is a free $\ZZ$-module of rank equal to
the total number of cells $A_{ij}$ in the affine paving.
The classes of the Zariski closures $[\overline{A_{ij}}]$ of these cells form a $\ZZ$-basis of
$H^{\bullet}(X;\ZZ)$.
\item For $1 \leq i \leq m$, the inclusion $\iota: X - X_i \hookrightarrow X$ induces a surjective ring
homomorphism on cohomology
\begin{equation}
\iota^*: H^{\bullet}(X; \ZZ) \twoheadrightarrow H^{\bullet}(X - X_i; \ZZ).
\end{equation}
\end{enumerate}

Let $X$ be a complex algebraic variety and let $V \twoheadrightarrow X$ be a complex vector bundle of rank $r$.
For $1 \leq i \leq r$ we have a {\em Chern class} $c_i(V) \in H^{2i}(X; \ZZ)$.
Introducing a variable $t$, the {\em Chern polynomial} $c_{\bullet}(V)$ is the generating function 
of these Chern classes:
\begin{equation}
c_{\bullet}(V) := 1 + c_1(V) t + c_2(V) t^2 + \cdots + c_r(V) t^r.
\end{equation}

Let $p: X' \rightarrow X$ be the fiber bundle whose fiber over a point $x \in X$ is the space of 
complete flags in the $r$-dimensional vector space $V_x$.
The induced map $p^*: H^{\bullet}(X;\ZZ) \rightarrow H^{\bullet}(X';\ZZ)$ is injective, so we may
view $H^{\bullet}(X;\ZZ)$ as a subring of $H^{\bullet}(X';\ZZ)$.
By the {\em Splitting Principle},
there exist (unique up to permutation) elements 
$x_1, \dots, x_r \in H^2(X'; \ZZ)$ such that 
$c_{\bullet}(V) = (1 + x_1 t) \cdots (1 + x_r t)$.  The degree 2 cohomology elements $x_1, \dots, x_r$ 
are  the {\em Chern roots} of $V$. We have 
$c_i(V) = e_i(x_1, \dots, x_r)$, where $e_i$ is the elementary symmetric polynomial of degree $i$.
Any symmetric polynomial in $x_1, \dots, x_r$ lies in $H^{\bullet}(X;\ZZ)$.

If $V$ and $W$ are two complex vector bundles over $X$, we can form the direct sum bundle
$V \oplus W$. The {\em Whitney sum formula} gives the following relation between the Chern polynomials
of these bundles:
\begin{equation}
c_{\bullet}(V \oplus W) = c_{\bullet}(V) \cdot c_{\bullet}(W).
\end{equation}

\section{Quotient rings}
\label{Quotient}

A subalgebra of the following quotient ring
will present the cohomology of $X_{\alpha,k}$.
It will be convenient to consider versions of this ring over both $\ZZ$ and $\QQ$.

\begin{defn}
\label{ring-definition}
Let $k > 0$ and let $\alpha = (\alpha_1, \dots, \alpha_r) \in [k]^r$ 
with $|\alpha| = n$.  Let $I_{\alpha,k} \subseteq \ZZ[x_1, \dots, x_n]$ be the 
ideal generated by
\begin{itemize}
\item the top $k$ elementary symmetric polynomials 
\begin{equation*}
e_n(\xx_n), e_{n-1}(\xx_n), \dots, e_{n-k+1}(\xx_n) 
\end{equation*}
in the full variable set
$\xx_n = (x_1, \dots, x_n)$ and
\item the complete homogeneous symmetric polynomials 
\begin{equation*}
h_{k-\alpha_i+1}(\xx_n^{(i)}), h_{k-\alpha_i+2}(\xx_n^{(i)}), \dots, h_k(\xx_n^{(i)})
\quad (1 \leq i \leq r)
\end{equation*}
in the $i^{th}$ batch of variables
$\xx_n^{(i)} = (x_{\alpha_1 + \cdots + \alpha_{i-1} + 1}, \dots, x_{\alpha_1 + \cdots + \alpha_{i-1} + \alpha_i})$.
\end{itemize}
Write $R_{\alpha,k} := \ZZ[\xx_n]/I_{\alpha,k}$ for the corresponding quotient ring.

Let $J_{\alpha,k} \subseteq \QQ[\xx_n]$ be the ideal in the polynomial ring with rational coefficients
having the same generators as $I_{\alpha,k}$. Write $S_{\alpha,k} := \QQ[\xx_n]/J_{\alpha,k}$ 
for the corresponding quotient.
\end{defn}

Both of the rings $R_{\alpha,k}$ and $S_{\alpha,k}$ are graded algebras (over $\ZZ$ and $\QQ$, respectively)
and carry a graded action of the parabolic subgroup $\symm_{\alpha}$ of $\symm_n$.
These rings are related by extension of coefficients: $S_{\alpha,k} = \QQ \otimes_{\ZZ} R_{\alpha,k}$.

Although the $\ZZ$-algebra $R_{\alpha,k}$ will be more useful for cohomology presentation, 
the $\QQ$-algebra $S_{\alpha,k}$ is easier to study directly, and
 $S_{\alpha,k}$ is the object of study in this section.
In Section~\ref{Cohomology} we will apply  Gr\"obner theory and a Demazure character identity from \cite{HRS}
to bootstrap  results to $R_{\alpha,k}$.

\subsection{A slightly unusual $\symm_n$-action}
The symmetric group $\symm_n$ acts on the set $\ZZ_{\geq 0}^n$ of length $n$ words over the 
nonnegative integers by the rule
\begin{equation}
\label{classical-word-action}
\pi.(a_1 \dots a_n) := a_{\pi(1)} \dots a_{\pi(n)}, \quad \pi \in \symm_n, \, \, a_1 \dots a_n \in \ZZ_{\geq 0}^n.
\end{equation}
In our analysis of the ring $S_{\alpha,k}$, we will use a less standard action of $\symm_n$ on 
$\ZZ_{\geq 0}^n$.

For $1 \leq i \leq n-1$, let $s_i = (i,i+1) \in \symm_n$ be the corresponding adjacent transposition.
For any word $a_1 \dots a_n \in \ZZ_{\geq 0}^n$, we define a new word
$s_i \cdot  (a_1 \dots a_n)$ by the rule
\begin{equation}
\label{unusual-word-action}
s_i \cdot (a_1 \dots a_i a_{i+1} \dots a_n) := \begin{cases}
a_1 \dots a_{i+1} (a_i - 1) \dots a_n & \text{if $a_i > a_{i+1}$,} \\
a_1 \dots (a_{i+1} + 1) a_{i} \dots a_n & \text{if $a_i \leq a_{i+1}$.}
\end{cases}
\end{equation}
For example, if $n = 4$ and $a_1 a_2 a_3 a_4 = 3225$ then
\begin{equation*}
s_1 \cdot (3225) = 2225, \quad s_2 \cdot(3225) = 3325, \quad s_3 \cdot (3225) = 3262.
\end{equation*}

\begin{proposition}
\label{unusual-proposition}
The rule \eqref{unusual-word-action} extends to give a free action of $\symm_n$ on the 
set of all words in $\ZZ_{\geq 0}^n$. Furthermore, every $\symm_n$-orbit in $\ZZ_{\geq 0}^n$ under
this action contains a unique strictly decreasing word.
\end{proposition}

The action $\pi \cdot (-)$ afforded by 
\eqref{unusual-word-action} is  different from the more traditional action
$\pi.(-)$.
In particular, the action $\pi \cdot (-)$ is free whereas $\pi.(-)$ is not.

\begin{proof}
These assertions become evident when we look at the rule \eqref{unusual-word-action} in a different way.
For $n \geq 0$, let $\BBB_n$ be the family of words $b_1 \dots b_m$ (of any length $m \geq n$)
over the alphabet $\{1, 2, \dots, n, \infty \}$  such that $b_1 < \infty$, each of the positive letters $1, 2, \dots, n$
appears exactly once, and the letter $\infty$ appears $m-n$ times.
We describe a bijection between $\BBB_n$ and $\ZZ_{\geq 0}^n$.

Given a word $b_1 \dots b_m \in \BBB_n$, the {\em inversion code} is the sequence
$(i_1, \dots, i_n) \in \ZZ_{\geq 0}$ where for $1 \leq j \leq n$ the entry $i_j$ is given by
\begin{equation}
i_j := | \{ b_1, b_2, \dots, b_{r-1} \} \cap \{j+1, j+2, \dots, \infty \}|,  \quad \text{where $b_r = j$.}
\end{equation}
The number $i_j$ counts inversions in $b_1 \dots b_m$ ending at $j$.  For example, the 
inversion code of $(\infty, \infty, 2, 1, 3, \infty, \infty, \infty, 4) \in \BBB_4$ is $(3,2,2,5)$ and (interchanging $3$ and $4$) the inversion code of
$(\infty, \infty, 2, 1, 4, \infty, \infty, \infty, 3) \in \BBB_4$ is $(3,2,6,2)$.

Since we insist that words in $\BBB_n$ start with finite letters,
the association $b_1 \dots b_m \mapsto (i_1, \dots, i_n)$ of a word to its inversion code gives a bijection
$\BBB_n \xrightarrow{\sim} \ZZ_{\geq 0}^n$.  The set $\BBB_n$ carries a free action of $\symm_n$, where 
permutations $\pi \in \symm_n$ act on finite letters.
The formula \eqref{unusual-word-action} is the transfer of this action to $\ZZ_{\geq 0}^n$ under the 
inversion code bijection. The unique strictly decreasing word in a given $\symm_n$-orbit 
in $\ZZ_{\geq 0}^n$ is the unique word in the corresponding orbit of $\BBB_n$ in which the finite 
letters appear in decreasing order.
\end{proof}

\subsection{Orbit harmonics}
We want to describe a Gr\"obner basis for $J_{\alpha,k}$,
give the standard monomial basis of $S_{\alpha,k}$,
and determine the isomorphism type of $S_{\alpha,k}$
as an ungraded $\symm_{\alpha}$-module. 
Our basic technique is the method of {\em orbit harmonics};
see \cite{Bergeron} for a more leisurely exposition.

Let $Z \subseteq \QQ^n$ be a finite set of points.  We associate to $Z$ the ideal $\II(Z) \subseteq \QQ[\xx_n]$
of polynomials which vanish on $Z$:
\begin{equation}
\II(Z) := \{ f(\xx_n) \in \QQ[\xx_n] \,:\, f(\zz_n) = 0 \text{ for all $\zz_n \in Z$} \}.
\end{equation}
By Lagrange Interpolation, we can identify the quotient 
$\QQ[\xx_n]/\II(Z)$ with the $\QQ$-vector space $\QQ[Z]$ of all functions $Z \rightarrow \QQ$.
In particular, we have $\dim( \QQ[x_1, \dots, x_n]/\II(Z) ) = |Z|$.

The ideal $\II(Z)$ is typically inhomogeneous.  To get a homogeneous ideal we proceed as follows.
For any nonzero polynomial $f \in \QQ[\xx_n]$ let $f = f_d + \cdots + f_1 + f_0$
where $f_i$ is homogeneous of degree $i$ and $f_d \neq 0$.
Write $\tau(f) := f_d$ for the top degree component of $f$.  Define an ideal $\TT(Z) \subseteq \QQ[x_1, \dots, x_n]$
by the rule
\begin{equation}
\TT(Z) := \langle \tau(f) \,:\, f \in \II(Z), \, \, f \neq 0 \rangle.
\end{equation}

The ideal $\TT(Z)$ is homogeneous by construction, so that 
$\QQ[\xx_n]/\TT(Z)$ is a graded $\QQ$-algebra.
We have isomorphisms of ungraded $\QQ$-vector spaces
\begin{equation}
\label{orbit-harmonics-isomorphisms}
\QQ[\xx_n]/\TT(Z) \cong \QQ[x_1, \dots, x_n]/\II(Z) \cong \QQ[Z].
\end{equation}

Let $G \subseteq GL_n(\QQ)$ be any finite matrix group acting on $\QQ^n$ canonically
and on $\QQ[\xx_n]$ by linear substitution.
When the point set $Z \subseteq \QQ^n$ is stable under the action of $G$,
the ideals $\II(Z)$ and $\TT(Z)$ are also $G$-stable and
 \eqref{orbit-harmonics-isomorphisms} is a sequence of 
ungraded $G$-module isomorphisms.
In our setting, the group $G$ will  be the parabolic subgroup $\symm_{\alpha}$
of $\symm_n$, viewed as a group of permutation matrices.
Our point locus $Z$ is given as follows.

\begin{defn}
\label{point-locus-definition}  Let $\alpha \in [k]^r$ be a composition with $|\alpha| = n$.
Fix $k$ distinct rational numbers $\eta_1, \dots, \eta_k$.  Let $Z_{\alpha,k} \subseteq \QQ^n$
be the set of points $(z_1, \dots, z_n)$ such that
\begin{itemize}
\item $\{z_1, \dots, z_n \} = \{\eta_1, \dots, \eta_k\}$ and
\item for $1 \leq i \leq r$ the $i^{th}$ batch of coordinates 
$z_{\alpha_1 + \cdots + \alpha_{i-1} + 1}, \dots, z_{\alpha_1 + \cdots + \alpha_{i-1} + \alpha_i}$
are distinct.
\end{itemize}
\end{defn}

The point set $Z_{\alpha,k}$ is stable under the action of $\symm_{\alpha}$.
There is a natural correspondence between $\OP_{\alpha,k}$ and $Z_{\alpha,k}$; in the case
$\alpha = (2,1,2), k = 3$ we have
\begin{equation*}
(2 \, 3 \mid 4 \mid 1 \, 5 ) \leftrightarrow (\eta_3, \eta_1, \eta_1, \eta_2, \eta_3).
\end{equation*}
Under this correspondence we have the identification of ungraded $\symm_{\alpha}$-modules 
\begin{equation}
\QQ[\OP_{\alpha,k}] = \QQ[Z_{\alpha,k}].
\end{equation}
We get a lower bound for the dimension of $S_{\alpha,k}$.

\begin{lemma}
\label{first-containment}
Let $\alpha \in [k]^r$ satisfy $|\alpha| = n$. 
We have $J_{\alpha,k} \subseteq \TT(Z_{\alpha,k})$.
\end{lemma}

We will prove the equality $J_{\alpha,k} = \TT(Z_{\alpha,k})$ in the next subsection.

\begin{proof}
We show that each generator of $J_{\alpha,k}$ lies in $\TT(Z_{\alpha,k})$.
We begin with the generators of the form $e_{d}(x_1, \dots, x_n)$ for $n-k+1 \leq d \leq n$.
Introduce a new variable $t$ and consider the rational function
\begin{equation}
\label{first-rational-function}
\frac{(1 + x_1 t) (1 + x_2 t) \cdots (1 + x_n t)}{(1 + \eta_1 t) (1 + \eta_2 t) \cdots (1 + \eta_k t)}.
\end{equation}
When $(x_1, \dots, x_n) \in Z_{\alpha,k}$, the $k$ factors in the denominator cancel with $k$ factors
in the numerator to give a polynomial in $t$ of degree $n-k$.  For any $n-k+1 \leq d \leq n$, the coefficient
of $t^d$ in this expression is then
\begin{equation}
0 = \sum_{j = 0}^d  (-1)^j e_{d-j}(\xx_n) h_j(\eta_1, \dots, \eta_k).
\end{equation}
Taking the top degree component, we see that $e_d(\xx_n) \in \TT(Z_{\alpha,k})$.

Now fix $1 \leq i \leq r$ and consider the rational function
\begin{equation}
\label{second-rational-function}
\frac{(1 - \eta_1 t) (1 - \eta_2 t) \cdots (1 - \eta_k t)}
{(1 - x_{\alpha_1 + \cdots + \alpha_{i-1} + 1} t) \cdots (1 + x_{\alpha_1 + \cdots + \alpha_{i-1} + \alpha_i} t)}
\end{equation}
involving the $i^{th}$ batch of variables $\xx_n^{(i)}$.
When $(x_1, \dots, x_n) \in Z_{\alpha,k}$, the $\alpha_i$ factors in the denominator cancel with $\alpha_i$ factors
in the numerator to give a polynomial in $t$ of degree $k - \alpha_i$.  In particular, for any $d > k - \alpha_i$,
the coefficient of $t^d$ in this polynomial is
\begin{equation}
0 = \sum_{j = 0}^k (-1)^j 
h_{d - j}(\xx_n^{(i)})
e_j(\eta_1, \dots, \eta_k).
\end{equation}
Taking the top degree component, we get
$h_{d}(\xx_n^{(i)}) 
\in \TT(Z_{\alpha,k})$.
\end{proof}

Lemma~\ref{first-containment} tells us that $\dim(S_{\alpha,k}) \geq |\OP_{\alpha,k}|$.
To get the reverse inequality, we will need the unusual action of $\symm_n$
(or rather its parabolic subgroup $\symm_{\alpha}$) afforded
by Proposition~\ref{unusual-proposition}.

\subsection{The structure of $S_{\alpha,k}$}
For any subset $S = \{ s_1 < s_2 < \cdots < s_t \} \subseteq [n]$, the 
{\em skip sequence} 
$\gamma(S) = (\gamma_1, \dots, \gamma_n)$ is defined by
\begin{equation}
\gamma_i = \begin{cases}
i - j + 1 & \text{if $i = s_j \in S$,} \\
0 & \text{if $i \notin S$.}
\end{cases}
\end{equation}
The {\em reverse skip sequence} is 
$\gamma(S)^* := (\gamma_n, \dots, \gamma_1)$.
For example, if $n = 8$ and $S = \{2,3,5,8\}$ then
$\gamma(S) = (0,2,2,0,3,0,0,5)$ and
$\gamma(S)^* = (5,0,0,3,0,2,2,0)$.
We also let $\xx(S)^* := x_1^{\gamma_n} \cdots x_n^{\gamma_1}$ be the 
{\em reverse skip monomial}, so that 
$\xx(S)^* = x_1^5 x_4^3 x_6^2 x_7^2$ in our example.
The terminology comes from the fact that nonzero entries of skip sequences $\gamma(S)$
increment whenever the set $S$ skips an element.

\begin{defn}
\label{nonskip-monomials}
Let $\alpha \in [k]^r$ satisfy $|\alpha| = n$ and let $m \in \QQ[x_1, \dots, x_n]$ be a monomial.
The monomial $m$ is {\em $\alpha$-nonskip} if 
\begin{itemize}
\item for any subset $S \subseteq [n]$ with $|S| = n-k+1$ we have 
$\xx(S)^* \nmid m$, and
\item for $1 \leq i \leq r$ and $1 \leq j \leq \alpha_i$, the $j^{th}$ variable in the $i^{th}$ batch
of variables satisfies
$x_{\alpha_1 + \cdots + \alpha_{i-1} + j}^{k - j + 1} \nmid m$.
\end{itemize}

Let $\MMM_{\alpha,k}$ be the set of $\alpha$-nonskip monomials in $\QQ[x_1, \dots, x_n]$.
\end{defn}

The set $\MMM_{\alpha,k}$ will turn out to be the standard monomial basis for 
$S_{\alpha,k}$ under the $\neglex$ term order.
To show that $\MMM_{\alpha,k}$ contains this standard monomial basis, we 
quote the following result of Haglund, the author, and Shimozono \cite{HRS}.
This polynomial identity will be crucial for our work.

\begin{theorem}
\label{demazure-identity}
(Haglund-R.-Shimozono \cite[Lem. 3.4]{HRS})
Let $S \subseteq [n]$ satisfy $|S| = n-k+1$.  We have the following identity involving Demazure characters:
\begin{equation}
\kappa_{\gamma(S)^*}(\xx_n) = 
\sum_{\lambda} (-1)^{|\lambda|} \kappa_{\overline{\gamma(S)^*} - \lambda}(\xx_n)
e_{n-k+1+|\lambda|}(\xx_n).
\end{equation}
Here $\gamma(S)^* = (\gamma_n, \dots, \gamma_1)$ is the reverse skip sequence of $S$,
the sequence $\overline{\gamma(S)^*}$ is the componentwise maximum of the length $n$
sequences $(\gamma_n - 1, \dots, \gamma_1 - 1)$ and the length $n$ sequence of zeroes
$(0, \dots, 0)$, and the sum is over subsets $\lambda$ of the skyline diagram of $\gamma(S)^*$
which are {\em left leaning}.
\end{theorem}

We will not need the notion of a skyline diagram, or a left leaning subset thereof, in this paper.
It suffices for us to notice that 
\begin{quote}
{\em if $\alpha \in [k]^r$ with $|\alpha| = n$, the Demazure character $\kappa_{\gamma(S)^*}(\xx_n)$
corresponding to the reverse skip sequence $\gamma(S)^*$ may be expressed as a 
$\ZZ[\xx_n]$-linear 
combination of $e_n(\xx_n), e_{n-1}(\xx_n), \dots, e_{n-k+1}(\xx_n)$.}
\end{quote}
It is this central fact that will allow us to 
\begin{itemize}
\item describe the Gr\"obner basis and ungraded $\symm_{\alpha}$-isomorphism type of $S_{\alpha,k}$ and
\item present the {\em integral} (rather than just rational) cohomology $H^{\bullet}(X_{\alpha,k}; \ZZ)$
as the $\symm_{\alpha}$-invariant subring $(R_{\alpha,k})^{\symm_{\alpha}}$ of $R_{\alpha,k}$.
\end{itemize}
Theorem~\ref{demazure-identity} will help us 
go back and forth between $\QQ$-vector spaces and $\ZZ$-modules.

The proof of Theorem~\ref{demazure-identity} in \cite{HRS}
uses a sign-reversing involution and an extension of the dual Pieri rule to 
Demazure characters appearing in a 2011 paper of Hagund, Luoto, Mason, and
van Willigenburg \cite{HLMV}.
Given the elegant definition of the spaces $X_{\alpha,k}$ of spanning configurations, it is natural
to wonder whether their integral cohomology could be presented using pre-2011 technology.

Here is our first application of Theorem~\ref{demazure-identity}.

\begin{lemma}
\label{standard-monomial-containment}
Let $\alpha \in [k]^r$ with $|\alpha| = n$.  The following monomials are contained in the initial
ideal $\initial_<(J_{\alpha,k})$ under the $\neglex$ term order:
\begin{itemize}
\item the reverse skip monomial $\xx(S)^*$, where $S \subseteq [n]$ satisfies $|S| = n-k+1$ and
\item the variable power $x_{\alpha_1 + \cdots + \alpha_{i-1} + j}^{k - j + 1}$ involving the $j^{th}$ variable
from the $i^{th}$ batch of variables, where $1 \leq i \leq r$ and $1 \leq j \leq \alpha_i$.
\end{itemize}
Consequently, the set $\MMM_{\alpha,k}$ of $\alpha$-nonskip monomials contains the standard monomial
basis of $S_{\alpha,k}$.
\end{lemma}

\begin{proof}
Let $S \subseteq [n]$ satisfy $|S| = n-k+1$.
By Theorem~\ref{demazure-identity}, the Demazure character 
$\kappa_{\gamma(S)^*}(x_1, \dots, x_n)$ lies in the ideal $J_{\alpha,k}$.
Since the $\neglex$-leading term of 
$\kappa_{\gamma(S)^*}(x_1, \dots, x_n)$ is $\xx(S)^*$, we have the first bullet point.
For the second bullet point, the identity
\begin{equation}
h_d(y_1, \dots, y_m) = h_d(y_1, \dots, y_{m-1}) + y_m \cdot h_{d-1}(y_1, \dots, y_m), \quad m, d > 0
\end{equation}
may be used to deduce that 
\begin{equation}
h_{k-j+1}(x_{\alpha_1 + \cdots + \alpha_{i-1} + 1}, \dots, x_{\alpha_1 + \cdots + \alpha_{i-1} + j}) \in J_{\alpha,k}
\end{equation}
for each $1 \leq j \leq \alpha_i$.  
Taking the $\neglex$-leading term, we see that 
$x_{\alpha_1 + \cdots + \alpha_{i-1} + j}^{k - j + 1} \in \initial_<(J_{\alpha,k})$.
\end{proof}

Our next task is to prove that $|\MMM_{\alpha,k}| = |\OP_{\alpha,k}|$ for any $\alpha \in [k]^r$.
There are two main ingredients to our argument: a notion of {\em coinversion code}
due to the author and Wilson \cite{RW} and the unusual $\symm_n$-action of
Proposition~\ref{unusual-proposition}.
We first prove that the
restriction of this 
action to $\symm_{\alpha}$ preserves (the exponent sequences of) monomials in $\MMM_{\alpha,k}$.

\begin{lemma}
\label{unusual-preservation}
Let $\alpha \in [k]^r$ with $|\alpha| = n$.
Identify monomials $x_1^{a_1} \cdots x_n^{a_n}$ in $\QQ[\xx_n]$ with their exponent
sequences $(a_1, \dots, a_n)$.
The set of monomials
$\MMM_{\alpha,k}$ is stable under the action of the parabolic subgroup $\symm_{\alpha} \subseteq \symm_n$
given in Proposition~\ref{unusual-proposition}.
\end{lemma}

\begin{proof}
Consider an index $p \in [n-1] - \{\alpha_1, \alpha_1 + \alpha_2, \dots \}$ corresponding to 
a typical adjacent transposition $s_p = (p, p+1)$ generating $\symm_{\alpha}$.
Suppose $m = x_1^{a_1} \cdots x_n^{a_n} \in \MMM_{\alpha,k}$ is $\alpha$-nonskip. We show
$s_p \cdot m \in \MMM_{\alpha,k}$ is  also $\alpha$-nonskip.
The application of $s_p \cdot ( - )$ sends $(a_p, a_{p+1})$ to $(a'_p, a'_{p+1})$ where
\begin{equation}
\begin{cases}
 (a'_p, a'_{p+1}) = (a_{p+1}, a_p - 1) & \text{if $a_p > a_{p+1}$,} \\
 (a'_p, a'_{p+1}) = (a_{p+1} + 1, a_p) &  \text{if $a_p \leq a_{p+1}$}
\end{cases}
\end{equation}
while leaving the remaining entries of $(a_1, \dots, a_n)$ unchanged.

Suppose $x_p$ lies in the $i^{th}$ batch of variables, so that 
$x_p = x_{\alpha_1 + \cdots + \alpha_{i-1} + j}$ for some $1 \leq j < \alpha_i$.
Since $m \in \MMM_{\alpha,k}$,
we have $a_p < k - j + 1$ and $a_{p+1} < k - j$.
Checking both possibilities $a_p > a_{p+1}$ and $a_p \leq a_{p+1}$, we see that
$a'_p < k - j + 1$ and $a'_{p+1} < k - j$.
This proves that the monomial $s_j \cdot m$ satisfies the second bullet point of 
Definition~\ref{nonskip-monomials}.

For the first bullet point of Definition~\ref{nonskip-monomials}, write 
$s_p \cdot (a_1, \dots, a_n) = (a'_1, \dots, a'_n)$
and suppose we have the componentwise inequality
$\gamma(S')^* \leq (a'_1, \dots, a'_n)$ for some subset $S' \subseteq [n]$ with $|S'| = n-k+1$.
Define a new subset $S \subseteq [n]$ by
\begin{equation}
S := \begin{cases}
(S' - \{p\}) \cup \{p+1\} & \text{if $p \in S'$ and $p+1 \notin S'$,} \\
(S' - \{p+1\}) \cup \{p\} & \text{if $p +1 \in S'$ and $p \notin S'$,} \\
S' & \text{otherwise.} \\
\end{cases}
\end{equation}
One checks the componentwise inequality $\gamma(S) \leq (a_1, \dots, a_n)$, which contradicts
the assumption that $m$ is $\alpha$-nonskip.
We conclude that $s_j \cdot m$ is $\alpha$-nonskip, as desired.
\end{proof}

We are ready to show the equality $|\MMM_{\alpha,k}| = |\OP_{\alpha,k}|$.  
Lemma~\ref{unusual-preservation} is a key tool in this argument.

\begin{lemma}
\label{cardinality-coincidence}
Let $\alpha \in [k]^r$ with $|\alpha| = n$.  We have $|\MMM_{\alpha,k}| = |\OP_{\alpha,k}|$.
\end{lemma}

\begin{proof}
Let $\sigma = (B_1 \mid \cdots \mid B_k) \in \OP_{n,k}$ be an ordered set partition of $[n]$
with $k$ blocks.
The {\em coinversion code} $\code(\sigma) = (c_1, \dots, c_n)$ of $\sigma$ is the length $n$ 
sequence whose $i^{th}$ term is
\begin{equation}
c_i = \begin{cases}
| \{ j' > j \,:\, \min(B_{j'}) > i \} | & \text{if $i = \min(B_j)$} \\
| \{ j' > j \,:\, \min(B_{j'}) > i \} |  + (j-1) & \text{if $i \in B_j$ and $i > \min(B_j)$.}
\end{cases}
\end{equation}

The author and Wilson prove \cite[Thm. 2.2]{RW} that the map $\sigma \mapsto \code(\sigma)$ 
is a bijection from $\OP_{n,k}$ to sequences $(c_1, \dots, c_n) \in [k]^n$ such that 
for any $S \subseteq [n]$ with $|S| = n-k+1$, the componentwise inequality
$\gamma(S)^* \leq (c_1, \dots, c_n)$ does {\em not} hold.

Identifying monomials with their exponent vectors, if 
the map $\sigma \mapsto \code(\sigma)$ restricted to give a 
bijection from $\OP_{\alpha,k}$ to $\MMM_{\alpha,k}$, this would prove the result. 
Unfortunately, this is not the case.
For example, when $\alpha = (2,2)$ and $k =3$ we have
$\code(3 \, 4 \mid 1 \mid 2) = (1,0,0,0)$.
Since $3$ and $4$ lie in the same block, $(3 \, 4 \mid 2 \mid 1) \notin \OP_{(2,2),3}$ but
$x_1^1 x_2^0 x_3^0 x_4^0 \in \MMM_{\alpha,k}$.

We combine coinversion codes and the action of $\symm_{\alpha}$ on $\MMM_{\alpha,k}$ 
coming from Lemma~\ref{unusual-preservation} to prove the result.
A length $n$ sequence $(c_1, \dots, c_n)$ is {\em $\alpha$-decreasing} if for $1 \leq i \leq r$
the $i^{th}$ batch of terms is strictly decreasing:
\begin{equation*}
c_{\alpha_1 + \cdots + \alpha_{i-1} + 1} > c_{\alpha_1 + \cdots + \alpha_{i-1} + 2} >
\cdots > c_{\alpha_1 + \cdots + \alpha_{i-1} + \alpha_i}.
\end{equation*}
For example, the sequence $(5,3,0,0,1)$ is $(3,1,1)$-decreasing but not $(1,3,1)$-decreasing or
$(1,1,3)$-decreasing.
We call  a monomial $x_1^{c_1} \cdots x_n^{c_n}$  
$\alpha$-decreasing if $(c_1, \dots c_n)$ is $\alpha$-decreasing.
Let $\overline{\MMM_{\alpha,k}} \subseteq \MMM_{\alpha,k}$ be the set of $\alpha$-decreasing
$\alpha$-nonskip monomials.

{\bf Claim:}
{\em The coinversion code bijection $\sigma \mapsto \code(\sigma)$ puts $\overline{\MMM_{\alpha,k}}$
in a one-to-one correspondence with a subset $\overline{\OP_{\alpha,k}}$ of $\OP_{\alpha,k}$.
Furthermore, every element of $\OP_{\alpha,k}$ is $\symm_{\alpha}$-conjugate to a unique element of 
$\overline{\OP_{\alpha,k}}$.}

Let us explain how the lemma follows from the claim. The action of $\symm_{\alpha}$ on 
$\OP_{\alpha,k}$ is free by the definition of $\OP_{\alpha,k}$. 
The action of $\symm_{\alpha}$ on $\MMM_{\alpha,k}$ coming from
Lemma~\ref{unusual-preservation} is free by Proposition~\ref{unusual-proposition}.
Also by Proposition~\ref{unusual-proposition}, every monomial in $\MMM_{\alpha,k}$
is $\symm_{\alpha}$-conjugate to a unique monomial in $\overline{\MMM_{\alpha,k}}$.
Given the claim, we can use these free actions and the bijection 
$\overline{\MMM_{\alpha,k}} \leftrightarrow \overline{\OP_{\alpha,k}}$ 
to uniquely construct a $\symm_{\alpha}$-equivariant
bijection between $\MMM_{\alpha,k} \leftrightarrow \OP_{\alpha,k}$.

Now we prove the claim. 
We need to understand the inverse of the coinversion word bijection $\sigma \mapsto \code(\sigma)$.
The inverse is constructed recursively by an insertion algorithm;
following the notation of \cite[Proof of Thm. 2.2]{RW}, we call this inverse map $\iota$.

Let 
$(B_1 \mid \cdots \mid B_k)$ be a sequence of
of $k$ possibly empty sets of positive integers. We define the {\em coinversion label}
of the sets $B_1, \dots, B_k$ by labeling the empty sets with $0, 1, \dots, j$ from right to left (where there are $j+1$
empty sets), and then labeling the nonempty sets with $j+1, j+2, \dots, k-1$ from left to right. 
An example of coinversion labels is as follows, displayed as superscripts:
\begin{equation*}
(  \varnothing^2 \mid 1 \,  3^3 \mid \varnothing^1 \mid 2 \,  5^4 \mid 4^5 \mid \varnothing^0).
\end{equation*}
By construction, each of the letters $0, 1, \dots, k-1$ appears exactly once as a coinversion label.

Let $(c_1, \dots, c_n) \in [k]^n$ satisfy $\gamma(S)^* \not\leq (c_1, \dots, c_n)$ for all $S \subseteq [n]$
with $|S| = n-k+1$.
We define $\iota(c_1, \dots, c_n) = (B_1 \mid \cdots \mid B_k)$ recursively by starting with the sequence 
$(\varnothing \mid \cdots \mid \varnothing)$ of $k$ copies of the empty set, 
and for $i = 1, 2, \dots, n$ 
inserting $i$ into the unique block with coinversion label $c_i$.

Let $n = 9, k = 4$, and  
$(c_1, \dots, c_9) = (2,0,3,1,0,0,2,1,0)$.
Observe that $(c_1, \dots, c_9)$ is $\alpha$-decreasing where $\alpha = (2,3,1,3)$.  We 
calculate $\iota(c_1, \dots, c_9)$ with the following table.
\begin{center}
\begin{tabular}{c | c | c}
$i$ & $c_i$ & $\sigma$  \\ \hline
$1$ & $2$ & $(\varnothing^3 \mid \varnothing^2 \mid \varnothing^1 \mid \varnothing^0 )$ \\
$2$ & $0$ & $(\varnothing^2 \mid 1^3 \mid \varnothing^1 \mid \varnothing^0 )$  \\
$3$ & $3$ & $(\varnothing^1 \mid 1^2 \mid \varnothing^0 \mid 2^3 )$ \\
$4$ & $1$ & $(\varnothing^1 \mid 1^2 \mid \varnothing^0 \mid 2 \, 3^3 )$ \\
$5$ & $0$ & $(4^1 \mid 1^2 \mid \varnothing^0 \mid 2 \, 3^3 )$ \\
$6$ & $0$ & $(4^0 \mid 1^1 \mid 5^2 \mid 2 \, 3^3 )$\\
$7$ & $2$ & $(4 \, 6^0 \mid 1^1 \mid 5^2 \mid 2 \, 3^3 )$ \\
$8$ & $1$ &  $(4 \, 6^0 \mid 1^1 \mid 5\, 7^2 \mid 2 \, 3^3 )$ \\
$9$ & $0$ & $(4 \, 6^0 \mid 1 \, 8^1 \mid 5\, 7^2 \mid 2 \, 3^3 )$ \\
\end{tabular}
\end{center}
We conclude  
$\iota(2,0,3,1,0,0,2,1,0) = (4 \, 6 \, 9 \mid 1 \, 8 \mid 5\, 7 \mid 2 \, 3 )$.
It is proven in \cite[Proof of Thm. 2.2]{RW} that $\iota(c_1, \dots, c_n)$ is always an ordered set partition
(i.e. that there are no empty blocks at the end of this procedure).
Observe that the output  $\iota(c_1, \dots, c_n)$ lies in $\OP_{\alpha,k}$.

Assume that the sequence $(c_1, \dots, c_n)$ is both $\alpha$-decreasing and $\alpha$-nonskip.
When the algorithm defining $\iota$ processes the $i^{th}$ batch of letters in $(c_1, \dots, c_n)$
(for some $1 \leq i \leq r$), the $\alpha$-decreasing condition makes the algorithm place 
the $i^{th}$ batch of numbers in $[n]$ in different 
blocks (see the above example); this shows that $\iota(c_1, \dots, c_n) \in \OP_{\alpha,k}$.  We  take
 $\overline{\OP_{\alpha,k}}$ to be the image of $\overline{\MMM_{\alpha,k}}$ under $\iota$.
 
 It remains to check that every ordered set partition in $\OP_{\alpha,k}$ is $\symm_{\alpha}$-conjugate
 to a unique element in $\overline{\OP_{\alpha,k}}$.  

 What do elements of $\overline{\OP_{\alpha,k}}$ look like?
 For any element $\sigma = (B_1 \mid \cdots \mid B_k) \in \OP_{\alpha,k}$, consider the sequence 
 $\sigma^{(0)}, \sigma^{(1)}, \dots, \sigma^{(r)}$ of lists
 $\sigma^{(i)} = (B_1^{(i)} \mid \cdots \mid B_k^{(i)})$ of possibly empty sets of positive integers 
 given by
 \begin{equation*}
 B_j^{(i)} := B_j \cap \{1, 2, \dots, \alpha_1 + \cdots  + \alpha_i \}.
 \end{equation*}
 In particular, we have $\sigma^{(0)} = (\varnothing \mid \cdots \mid \varnothing)$
 and $\sigma^{(r)} = \sigma$.
 The sequence $\sigma^{(i)}$ is obtained from $\sigma^{(i-1)}$ by adding the $i^{th}$ batch of letters in $[n]$.
 The coinversion labels give a total order on the blocks of $\sigma^{(i-1)}$.
 The partition $\sigma$ is {\em $\alpha, i$-compatible} (for $1 \leq i \leq r$) if the transformation
 $\sigma^{(i-1)} \leadsto \sigma^{(i)}$ adds these letters (from smallest to largest) into blocks of 
 strictly decreasing coinversion label with respect to $\sigma^{(i-1)}$.  The partition $\sigma$ is 
 {\em $\alpha$-compatible} if it is $\alpha,i$-compatible for each $1 \leq i \leq r$.

To see how these concepts work, 
let $\alpha = ( 2,3,1,3)$ (so that $r = 4$) and consider our example
ordered set partition
$\sigma = (4 \, 6 \, 9 \mid 1 \, 8 \mid 5\, 7 \mid 2 \, 3 ) \in \OP_{9,4}$.  The sequences
 $\sigma^{(0)}, \sigma^{(1)}, \dots, \sigma^{(4)}$ (together with the relevant coinversion labels) are
 \begin{align*}
 \sigma^{(0)} &=  (\varnothing^3 \mid \varnothing^2 \mid \varnothing^1 \mid \varnothing^0 ) \\
 \sigma^{(1)} &= (\varnothing^1 \mid 1^2 \mid \varnothing^0 \mid 2^3 ) \\
 \sigma^{(2)} &=  (4^0 \mid 1^1 \mid 5^2 \mid 2 \, 3^3 ) \\
 \sigma^{(3)} &= (4 \, 6^0 \mid 1^1 \mid 5^2 \mid 2 \, 3^3 ) \\
 \sigma^{(4)} &= (4 \, 6 \, 9 \mid 1 \, 8 \mid 5\, 7 \mid 2 \, 3 )
 \end{align*}
 We see that the newly added elements at each stage are added in strictly decreasing order
 with respect to coinversion label, so that $\sigma$ is $\alpha$-compatible.
  From the definition of $\iota$
 we see that $\overline{\OP_{\alpha,k}} = \iota(\overline{\MMM_{\alpha,k}})$ 
 is precisely the set of $\alpha$-compatible elements in $\OP_{\alpha,k}$.

 Each partition $\sigma \in \OP_{\alpha,k}$ is $\symm_{\alpha}$-conjugate to a unique $\alpha$-compatible 
 partition.
Therefore, each partition $\sigma \in \OP_{\alpha,k}$ is 
$\symm_{\alpha}$-conjugate to a unique element of  $\overline{\OP_{\alpha,k}}$.
This completes the proof of the claim.
\end{proof}

We are ready to describe the standard monomial basis 
and $\symm_{\alpha}$-module structure of $S_{\alpha,k}$.

\begin{theorem}
\label{s-structure-theorem}
Let $\alpha \in [k]^r$ with $|\alpha| = n$.  Endow monomials in $\QQ[\xx_n]$ with
the $\neglex$ term order.  The set $\MMM_{\alpha,k}$ of $\alpha$-nonskip monomials is 
the standard monomial basis of $S_{\alpha,k}$.  We have an isomorphism of ungraded 
$\symm_{\alpha}$-modules
\begin{equation}
S_{\alpha,k} \cong_{\symm_{\alpha}} \QQ[\OP_{\alpha,k}].
\end{equation}
\end{theorem}

\begin{proof}
We have the chain of inequalities
\begin{equation}
\dim(S_{\alpha,k}) \leq |\MMM_{\alpha,k}| = |\OP_{\alpha,k}| \leq \dim(S_{\alpha,k}).
\end{equation}
The first inequality is Lemma~\ref{standard-monomial-containment},
the equality is Lemma~\ref{cardinality-coincidence}, and the 
second inequality is Lemma~\ref{first-containment}.
We conclude that $\dim(S_{\alpha,k}) = |\MMM_{\alpha,k}| = |\OP_{\alpha,k}|$.

Lemma~\ref{standard-monomial-containment} implies that $\MMM_{\alpha,k}$
is the standard monomial basis of $S_{\alpha,k}$.  Lemma~\ref{first-containment} and the fact that 
$\QQ[x_1, \dots, x_n]/\TT(Z_{\alpha,k}) \cong \QQ[\OP_{\alpha,k}]$ as ungraded $\symm_{\alpha}$-modules
imply that 
$J_{\alpha,k} = \TT(Z_{\alpha,k})$, which yields the desired $\symm_{\alpha}$-module isomorphism.
\end{proof}

We can also write down a Gr\"obner basis for $J_{\alpha,k}$. 
We state this result together with a fact about integrality we will need for cohomology
computation. 

\begin{proposition}
\label{groebner-proposition}
Let $\alpha \in [k]^r$ with $|\alpha| = n$.
Endow monomials in $\QQ[\xx_n]$ with the $\neglex$ term order.  A Gr\"obner basis 
for $J_{\alpha,k}$ consists of the Demazure characters 
\begin{equation*}
\kappa_{\gamma(S)^*}(\xx_n), \quad S \subseteq [n], \, \, |S| = n-k+1
\end{equation*}
together with the complete homogeneous symmetric polynomials
\begin{equation*}
h_{k-j+1}(x_{\alpha_1 + \cdots + \alpha_{i-1} + 1}, \dots, x_{\alpha_1 + \cdots + \alpha_{i-1} + j}), \quad
1 \leq i \leq r, \, \, 1 \leq j \leq \alpha_i.
\end{equation*}

Every element in the Gr\"obner basis above is a polynomial in $\ZZ[\xx_n]$ with integer coefficients
and $\neglex$-leading coefficient $1$.  Furthermore, every polynomial in the Gr\"obner basis above 
lies in the ideal $I_{\alpha,k} \subseteq \ZZ[\xx_n]$.
\end{proposition}

\begin{proof}
The proof of Lemma~\ref{standard-monomial-containment} shows that the polynomials in question lie
in $J_{\alpha,k}$ and Theorem~\ref{demazure-identity} guarantees 
 that they can be expressed as $\ZZ[\xx_n]$-linear combinations 
of the generators of $J_{\alpha,k}$.
The result follows from the identification of $\MMM_{\alpha,k}$ as the standard monomial basis of $S_{\alpha,k}$
in Theorem~\ref{s-structure-theorem}.
\end{proof}

\section{Mixed reduction}
\label{Mixed}
In this section we describe a linear algebra algorithm 
which will give an affine paving of the Grassmann product
$Gr(\alpha,k)$ which respects the inclusion
$X_{\alpha,k} \subseteq Gr(\alpha,k)$.
Since our algorithm will be a mixture of row and column reduction, we call it {\em mixed reduction}.
The results in this section are valid over any field $\FF$.

We model our moduli spaces using matrices. For any $\alpha \in [k]^r$ with $|\alpha| = n$, let 
$\VVV_{\alpha}$ be the family of $k \times n$  matrices given in block form by
\begin{equation}
\VVV_{\alpha} := \left\{
A = (A_1 \mid \cdots \mid A_r) \,:\,  \text{$A_i$ is a $k \times \alpha_i$ matrix and has full rank for 
$i = 1, 2, \dots, r$} 
\right\}.
\end{equation}
Let $\UUU_{\alpha} \subseteq \VVV_{\alpha}$ be the subfamily
\begin{equation}
\UUU_{\alpha} := \{ A \in \VVV_{\alpha} \,:\, \text{$A$ has full rank} \}
\end{equation}
of rank $k$ matrices in $\VVV_{\alpha}$.  Both of the sets $\VVV_{\alpha}$ and $\UUU_{\alpha}$
carry an action of the Levi subgroup $L_{\alpha} \subseteq GL_n$ on their columns. We have 
the identifications
\begin{equation}
\VVV_{\alpha}/L_{\alpha} = Gr(\alpha,k), \quad
\UUU_{\alpha}/L_{\alpha} = X_{\alpha,k}.
\end{equation}

\subsection{Permuted column reduced echelon form}
We review the classical decomposition of the Grassmannian $Gr(d,k)$ into Schubert cells.
Any point $V \in Gr(d,k)$ can be represented as the column space of a full rank $k \times d$ matrix $A$.
Applying the action of $GL_d$ on columns, we may obtain a unique such representation by putting $A$
into column reduced echelon form (CREF).  Given a $d$-element subset $I \subseteq [k]$, the 
{\em (open) Schubert cell} of $Gr(d,k)$ consists of all points $V \in Gr(d,k)$ represented by matrices 
whose CREF has pivot row set equal to $I$.
For $k = 4$ and $d = 2$, the Schubert cells are as follows, where the $\star$'s are arbitrary field elements.
\begin{scriptsize}
\begin{center}
\begin{tabular}{c c c c c c}
$\begin{pmatrix}
1 & 0 \\
0 & 1 \\
\star & \star \\
\star & \star 
\end{pmatrix}$ &
$\begin{pmatrix}
1 & 0 \\
\star & 0 \\
0 & 1 \\
\star & \star 
\end{pmatrix}$ & 
$\begin{pmatrix}
1 & 0 \\
\star & 0 \\
\star & 0 \\
0 & 1 
\end{pmatrix}$ &
$\begin{pmatrix}
0 & 0 \\
1 & 0 \\
0 & 1 \\
\star & \star
\end{pmatrix}$ &
$\begin{pmatrix}
0 & 0 \\
1 & 0 \\
\star & 0 \\
0 & 1 
\end{pmatrix}$ &
$\begin{pmatrix}
0 & 0 \\
0 & 0 \\
1 & 0 \\
0 & 1 
\end{pmatrix} $
 \\ & & & & & \\
$I = 12$ & $I = 13$ & $I = 14$ & $I = 23$ & $I = 24$ & $I = 34$
\end{tabular}
\end{center}
\end{scriptsize}

As $I$ varies over all $k$-element subsets of $[n]$, the  Schubert cells attached to $I$ induce an
affine paving of the Grassmannian $Gr(d,k)$.  Taking a product of such pavings gives a paving of
$Gr(\alpha,k)$.  Unfortunately, the product paving interacts poorly
with the inclusion $X_{\alpha,k} \subseteq Gr(\alpha,k)$: cells $C$ in the 
product paving can satisfy 
$C \not\subseteq X_{\alpha,k}$ and $C \cap X_{\alpha,k} \neq \varnothing$ simultaneously.
For this reason, we will need a different paving of $Gr(\alpha,k)$.
The first step in describing this new paving is a modification of the column reduction algorithm.

Let $A$ be a matrix with $k$ rows.
The usual column reduction algorithm prioritizes the rows of $A$ in the order $1, 2, \dots, k$.
If $\pi = \pi(1) \pi(2) \cdots \pi(k) \in \symm_k$ is any permutation, we could also
prioritize the rows 
in the order $\pi(1), \pi(2), \dots, \pi(k)$.  Let the output of this algorithm be the 
{\em $\pi$-column reduced echelon form ($\pi$-CREF)} of $A$. 
When $\pi = 12 \cdots k$ is the identity permutation, the $\pi$-CREF is the usual CREF.
We will only be interested in the case where the permutation $\pi$ has at most one descent.

To see an example of $\pi$-CREF, suppose
$k = 4$, $d = 2$, and $\pi = 2413 \in \symm_4$. 
The possible forms of the $\pi$-CREF of a full rank matrix, together with their pivot sets $I$, 
are
\begin{scriptsize}
\begin{center}
\begin{tabular}{cccccc}
$\begin{pmatrix}
\star & \star \\
1 & 0 \\
\star & \star \\
 0 & 1 
\end{pmatrix}$ &
$\begin{pmatrix}
0 & 1 \\
1 & 0 \\
\star & \star \\
\star &  0
\end{pmatrix}$ &
$\begin{pmatrix}
\star  & 0 \\
1 & 0 \\
0 & 1 \\
\star & 0
\end{pmatrix}$ &
$\begin{pmatrix}
0 & 1 \\
0 & 0 \\
\star  & \star \\
1 & 0
\end{pmatrix}$ &
$\begin{pmatrix}
\star & 0 \\
0 & 0 \\
0  & 1 \\
1 & 0 
\end{pmatrix}$ &
$\begin{pmatrix}
1 & 0 \\
0 & 0 \\
0 & 1 \\
0 & 0 
\end{pmatrix}$ \\
& & & & &  \\
$I = 24$ & $I = 12$ & $I = 23$ & $I = 14$ & $I = 34$ & $I = 13$
\end{tabular}
\end{center}
\end{scriptsize}

If $A$ is a $k \times d$ matrix for $d \leq k$ and $I \subseteq [k]$ is a subset of the rows of $A$,
let $\Delta_I(A)$ be the maximal minor of $A$ with row set $I$.
We record an important fact about the entries of a matrix in $\pi$-CREF.

\begin{lemma}
\label{minor-recovery}
Let $d \leq k$ and let $A$ be a full rank $k \times d$ matrix. Let $\pi \in \symm_k$ and let $B$ be the 
$\pi$-CREF of $A$.  Suppose that $B$ has pivots in row set $I \subseteq [k]$ (so that $|I| = d$).
For any position of $(r,s)$ of $B$ such that $r \notin I$, there exists a subset $J = J(r,s) \subseteq [k]$
such that $J$ differs from $I$ in exactly one element and 
\begin{center}
the $(r,s)$-entry of $B$ is the quotient of maximal
minors $\pm \Delta_J(A)/\Delta_I(A)$ of $A$,
\end{center}
where the sign depends  on $I, r, s,$ and $\pi$.
\end{lemma}

\begin{proof}
When $\pi$ is the identity permutation, this is the classical fact that the Pl\"ucker coordinates 
of any point in $Gr(d,k)$ contain (up to predictable signs and global scaling) the entries in the CREF of 
a $k \times d$ matrix representing that point.
This remains true if we order the columns of our matrices according to an arbitrary 
permutation $\pi$.
\end{proof}

\subsection{The mixed reduction algorithm}
Let $U \subseteq GL_k$ be the subgroup of {\em lower} triangular matrices with $1$'s on
 the diagonal.  Given a set sequence $\III = (I_1, \dots, I_r)$ in $[k]$,
 we define a subset $U(\III) \subseteq U$ according to the following sparsity pattern.
 
 \begin{defn}  
 Let $\III = (I_1, \dots, I_r)$ be a set
 sequence in $[k]$.  Let  $U(\III)$ consist of
 those matrices $(u_{i,j})_{1 \leq i,j \leq k} \in U$ such that for $1 \leq j < i \leq k$ we have 
 \begin{quote}
 $u_{i,j} = 0$ unless 
 there exists $1 \leq t \leq n$ such that $i \in I_t - (I_1 \cup \cdots \cup I_{t-1})$ and 
 $j \notin (I_1 \cup \cdots \cup I_{t-1} \cup I_t)$.
 \end{quote}
 \end{defn}

 When $k = 4$, two examples of $U(\III)$ -- one for a covering set sequence and one
 for a non-covering set sequence --  are as follows:
 \begin{scriptsize}
 \begin{equation*}
 U(24,34,3,14) = \begin{pmatrix}
 1 & 0 & 0 & 0 \\
0  & 1 & 0 & 0 \\
 0  & \star &  1 & 0 \\
0   & 0 & 0 &  1
 \end{pmatrix}  \quad
 U(14, 12, 1, 12) = \begin{pmatrix}
 1 & 0 & 0 & 0 \\
 \star & 1 & 0 & 0 \\
 \star & \star & 1 & 1 \\
 0 & 0 & 0 & 1
 \end{pmatrix}
 \end{equation*}
 \end{scriptsize}
The set $U(\III)$ is  a copy of affine space of dimension equal to the number
of $\star$'s.

Given a set sequence $\III = (I_1, \dots, I_r)$  in $[k]$, we also define a sequence 
of $r$ permutations in $\symm_k$.

\begin{defn}
Let $\III = (I_1, \dots, I_r)$ be set sequence in $[k]$.  Define $r$ permutations
$\pi^{(1)}_{\III}, \dots, \pi^{(r)}_{\III} \in \symm_k$ as follows.  The one-line notation $\pi^{(t)}_{\III}$ 
is obtained by listing the elements of $[k] - (I_1 \cup \cdots \cup I_{t-1})$ in increasing order,
then writing the elements of $I_1 \cup \cdots \cup I_{t-1}$ in increasing order.
We write $\pi^{(t)}$ instead of $\pi^{(t)}_{\III}$ when the set sequence is understood.
\end{defn}

If $\III = (24, 34, 3, 14)$ we get the permutation
sequence $\pi^{(1)} = 1234, \pi^{(2)} = 1324, \pi^{(3)} = 1234,$ and $\pi^{(4)} = 1234$.
For our other 
example $\III = (14,12,1,12)$ we get
the permutation sequence $\pi^{(1)} = 1234, \pi^{(2)} = 2314, \pi^{(3)} = 3124,$ and $\pi^{(4)} = 3124$.

Both the affine space $U(\III)$ and the permutation sequence
$\pi^{(1)}_{\III}, \dots, \pi^{(r)}_{\III}$ only depend on the initial unions of the sets in $\III$.

\begin{observation}
\label{jump-observation}
Let $\III = (I_1, \dots I_r)$ and $\III' = (I'_1, \dots, I'_r)$ be two set sequences in $[k]$ 
of the same type. Suppose that $I_1 \cup \cdots \cup I_t = I'_t \cup \cdots \cup I'_t$ for 
each $1 \leq t \leq r$. Then $U(\III) = U(\III')$ and
$\pi^{(t)}_{\III} = \pi^{(t)}_{\III'}$ for each $1 \leq t \leq r$.
\end{observation}

The unipotent group $U \subseteq GL_k$ and its subsets $U(\III)$ act on the rows of 
matrices in $\VVV_{\alpha}$ and
$\UUU_{\alpha}$.  Mixed reduction starts with a matrix $A \in \VVV_{\alpha}$ and outputs a set 
sequence $\III$ of type $\alpha$, a matrix $u \in U(\III)$,
and a matrix $B \in \VVV_{\alpha,k}$ fitting a certain sparsity pattern determined by $\III$
such that $A = uBg$ for some element $g \in L_{\alpha}$.
We describe the pattern satisfied by the output matrix $B$.

\begin{defn}
\label{pattern-definition}
Let $\III = (I_1, \dots, I_r)$ be a set sequence of type $\alpha \in [k]^r$ where $|\alpha| = n$.
Let $\pi^{(1)}, \dots, \pi^{(r)} \in \symm_k$ be the permutation sequence associated to $\III$.
The {\em pattern matrix} $\PM(\III)$ is a $k \times n$ matrix over the set $\{0, 1, \star \}$
with block form $\PM(\III) = (A_1 \mid \cdots \mid A_r)$ where $A_t$ is $k \times \alpha_t$.
For $1 \leq t \leq r$, the block $A_t$ is obtained as follows.

Let $C_t$ be the $\{0,1,\star\}$-matrix 
describing $\pi^{(t)}$-CREF with pivot set $I_t$.  
For any $j > i$ such that $i \in I_t - (I_1 \cup \cdots \cup I_{t-1})$ and
$j \in [k] - (I_1 \cup \cdots \cup I_{t-1} \cup I_t)$,
there will be a $\star$ in $C_t$ in row $j$ below the pivot in row $i$.
The block $A_t$ of $\PM(\III)$ is obtained from $C_t$ by replacing all such $\star$'s with $0$'s.
\end{defn}

We give the pattern matrices corresponding to our two example set sequences.
If $\III = (24, 34, 3, 14)$ 
the associated pattern matrix is
\begin{scriptsize}
\begin{equation*}
\PM(24, 34, 3, 13) = 
\begin{pmatrix}

\begin{matrix}
 0 & 0  \\
 1 &  0 \\
 0 & 0 \\
  0 & 1
\end{matrix} & \rvline &

\begin{matrix}
 0 & 0  \\
 \star & 0 \\
 1 & 0 \\
 0 & 1
\end{matrix} & \rvline &
 
\begin{matrix}
0 \\
0 \\
1 \\
\star
\end{matrix} & \rvline &

\begin{matrix}
 1 & 0  \\
 \star &  0 \\
 0 & 1 \\
  \star & \star
\end{matrix} 

\end{pmatrix} \\
\end{equation*}
\end{scriptsize}
The $(3,1)$-entry of the first block is $0$ instead of $\star$ because of the second paragraph
of Definition~\ref{pattern-definition}.
If $\III = (14,12,1,12)$ 
the associated pattern matrix is
\begin{scriptsize}
\begin{equation*}
\PM(14, 12, 1, 12) =
\begin{pmatrix}

\begin{matrix}
 1 & 0  \\
 0 &  0 \\
 0 & 0 \\
  0 & 1
\end{matrix} & \rvline &

\begin{matrix}
 0 & 1  \\
 1 & 0 \\
 0 & 0 \\
 \star & \star 
\end{matrix} & \rvline &
 
\begin{matrix}
1 \\
\star \\
0  \\
\star
\end{matrix} & \rvline &

\begin{matrix}
 1 & 0  \\
 0 &  1 \\
 0 & 0 \\
  \star & \star
\end{matrix} 

\end{pmatrix}
\end{equation*}
\end{scriptsize}
The $(2,1)$- and $(3,1)$-entries of the first block and the $(3,1)$-entry of the second block are $0$
instead of $\star$ because of the second paragraph of Definition~\ref{pattern-definition}.

A pattern matrix $\PM(\III)$ has a zero row if and only if the set sequence $\III$ fails to cover $[k]$.
For any set sequence $\III$, a matrix $B \in \VVV_{\alpha}$ is said to {\em fit the pattern} of $\III$
if $B$ can be obtained by replacing the $\star$'s in $\PM(\III)$ with field elements.

\begin{quote}
{\underline {\bf MIXED REDUCTION ALGORITHM}} \\
{\bf INPUT:} Integers $k, r > 0$, a length $r$ sequence $\alpha = (\alpha_1, \dots, \alpha_r) \in [k]^r$, and a matrix
$A \in \VVV_{\alpha}$ \\ 
{\bf OUTPUT:} A set sequence $\III = (I_1, \dots, I_r)$ in $[k]$ of type $\alpha$, an element 
$u \in U(\III)$, and a matrix $B \in \VVV_{\alpha}$ which fits the pattern of $\III$ such that  
$A = uBg$ for some element $g \in L_{\alpha}$ 
\begin{enumerate}
\item  Initialize a block matrix $B := (B_1 \mid \cdots \mid B_r)$ by 
$B_t := A_t$. Initialize $\III$ to be the empty set sequence. Initialize $u$ to be the $k \times k$ identity matrix.
\item For $t = 1, 2, \dots, r$ do the following,
\begin{enumerate}
\item Let $\pi^{(t)} = \pi^{(t)}(1) \dots \pi^{(t)}(k)$ 
be the permutation in $\symm_k$ obtained by writing the elements of 
$[k] - (I_1 \cup \cdots \cup I_{t-1})$ in increasing order, and then writing the elements of 
$I_1 \cup \cdots \cup I_{t-1}$ in increasing order.
\item Use column operations on $B_t$ to put $B_t$ into $\pi^{(t)}$-CREF. 
Let $I_t \subseteq [k]$ be the set of pivot rows 
obtained thereby. Append $I_t$ to the set sequence $\III$.  
\item  For any $j > i$ such that $i \in I_t - (I_1 \cup \cdots \cup I_{t-1})$ and 
$j \in [k] - (I_1 \cup \cdots \cup I_t)$, modify the $(j,i)$-entry of $u$ to clear the element of $B_t$ 
in row $j$ below the pivot in row $i$.
This performs a row operation
on the overall matrix $B$.
\end{enumerate}
\end{enumerate}
\end{quote}

We consider two examples of the mixed reduction algorithm.

\begin{example}
Take $k = 4$, $\alpha = (2,2,1,2)$, and start 
with the matrix
\begin{scriptsize}
\begin{equation*}
A = 
\begin{pmatrix}

\begin{matrix}
 0 & 0  \\
 1 &  0 \\
 -1 & 0 \\
  0 & 1
\end{matrix} & \rvline &

\begin{matrix}
 1 & 0  \\
 -1 & 0 \\
 2 & 0 \\
 0 & 1
\end{matrix} & \rvline &
 
\begin{matrix}
4 \\
-2 \\
8 \\
4 
\end{matrix} & \rvline &

\begin{matrix}
 1 & 1  \\
 2 &  2 \\
 -1 & 0 \\
  -1 & -1
\end{matrix} 

\end{pmatrix}.
\end{equation*}
\end{scriptsize}
We process the blocks of $A$ from left to right:
\begin{scriptsize}
$$
\begin{pmatrix}

\begin{matrix}
 0 & 0  \\
 1 &  0 \\
 -1 & 0 \\
  0 & 1
\end{matrix} & \rvline &

\begin{matrix}
 1 & 0  \\
 -1 & 0 \\
 2 & 0 \\
 0 & 1
\end{matrix} & \rvline &
 
\begin{matrix}
4 \\
-2 \\
8 \\
4 
\end{matrix} & \rvline &

\begin{matrix}
 1 & 1  \\
 2 &  2 \\
 -1 & 0 \\
  -1 & -1
\end{matrix} 
\end{pmatrix} \rightarrow
\begin{pmatrix}

\begin{matrix}
 0 & 0  \\
 1 &  0 \\
 0 & 0 \\
  0 & 1
\end{matrix} & \rvline &

\begin{matrix}
 1 & 0  \\
 -1 & 0 \\
 1 & 0 \\
 0 & 1
\end{matrix} & \rvline &
 
\begin{matrix}
4 \\
-2 \\
6 \\
4 
\end{matrix} & \rvline &

\begin{matrix}
 1 & 1  \\
 2 &  2 \\
 1 & 2 \\
  -1 & -1
\end{matrix} 
\end{pmatrix} \rightarrow
\begin{pmatrix}

\begin{matrix}
 0 & 0  \\
 1 &  0 \\
 0 & 0 \\
  0 & 1
\end{matrix} & \rvline &

\begin{matrix}
 1 & 0  \\
 -1 & 0 \\
 0 & 0 \\
 0 & 1
\end{matrix} & \rvline &
 
\begin{matrix}
4 \\
-2 \\
2 \\
4 
\end{matrix} & \rvline &

\begin{matrix}
 1 & 1  \\
 2 &  2 \\
 0 & 1 \\
  -1 & -1
\end{matrix} 
\end{pmatrix} \rightarrow
$$
$$
\begin{pmatrix}

\begin{matrix}
 0 & 0  \\
 1 &  0 \\
 0 & 0 \\
  0 & 1
\end{matrix} & \rvline &

\begin{matrix}
 1 & 0  \\
 -1 & 0 \\
 0 & 0 \\
 0 & 1
\end{matrix} & \rvline &
 
\begin{matrix}
2 \\
-1 \\
1 \\
2 
\end{matrix} & \rvline &

\begin{matrix}
 1 & 1  \\
 2 &  2 \\
 0 & 1 \\
  -1 & -1
\end{matrix} 
\end{pmatrix} \rightarrow
\begin{pmatrix}

\begin{matrix}
 0 & 0  \\
 1 &  0 \\
 0 & 0 \\
  0 & 1
\end{matrix} & \rvline &

\begin{matrix}
 1 & 0  \\
 -1 & 0 \\
 0 & 0 \\
 0 & 1
\end{matrix} & \rvline &
 
\begin{matrix}
2 \\
-1 \\
1 \\
2 
\end{matrix} & \rvline &

\begin{matrix}
 1 & 0  \\
 2 &  0 \\
 0 & 1 \\
  -1 & 2 
\end{matrix} 
\end{pmatrix}
$$
\end{scriptsize}
More explicitly, the steps are
\begin{enumerate}
\item The first block of $A$ is already in  
$\pi^{(1)}$-CREF where $\pi^{(1)} = 1234$.  We see that $I_1 = 24$.
\item We use  the pivot in row $2$ to cancel row $3$. 
\item The second block of $A$ is already in 
$\pi^{(2)}$-CREF where $\pi^{(2)} = 1324$. We see that $I_2 = 14$.
\item We use the pivot in row $1$ to cancel row $3$.
\item We find the $\pi^{(3)}$-CREF of the third block of $A$ where $\pi^{(3)} = 3124$. We see that $I_3 = 3$.
\item We find the $\pi^{(4)}$-CREF of the fourth block of $A$ where $\pi^{(4)} = 1234$. We see that $I_4 = 13$.
\item We use the  pivot in row $1$ to cancel rows $2$ and $4$. 
\end{enumerate}
The output set sequence $\III = (I_1, I_2, I_3, I_4)$ is $(24, 14, 3, 13)$.
The unipotent subset $U({\III})$ and the matrix $u \in U(\III)$ are
\begin{scriptsize}
\begin{equation*}
U({\III}) = \begin{pmatrix}
1 & 0 & 0 & 0 \\
0 & 1 & 0 & 0 \\
\star & \star & 1 & 0 \\
0 & 0 & 0 & 1 
\end{pmatrix} \quad
u = \begin{pmatrix}
1 & 0 & 0 & 0 \\
0 & 1 & 0 & 0 \\
1 & -1 & 1 & 0 \\
0 & 0 & 0 & 1 
\end{pmatrix}.
\end{equation*}
\end{scriptsize}
\end{example}

\begin{example}
Our second example applies mixed reduction to a matrix whose rank is not full.
Let 
\begin{scriptsize}
\begin{equation*}
A = 
\begin{pmatrix}

\begin{matrix}
 2 & 0  \\
 2 &  0 \\
 -2 & 0 \\
  1 & 1
\end{matrix} & \rvline &

\begin{matrix}
 0 & 1  \\
 1 & 1 \\
 1 & -1 \\
 2 & 2 
\end{matrix} & \rvline &
 
\begin{matrix}
-1 \\
0 \\
2  \\
-1
\end{matrix} & \rvline &

\begin{matrix}
  1 & 1 \\
  3 & 3 \\
  1  & 1 \\
  -1 & 0
\end{matrix} 
 \end{pmatrix}
\end{equation*}
\end{scriptsize}
Applying mixed reduction to $A$ yields the following.
\begin{scriptsize}
$$
\begin{pmatrix}

\begin{matrix}
 2 & 0  \\
 2 &  0 \\
 -2 & 0 \\
  1 & 1
\end{matrix} & \rvline &

\begin{matrix}
 0 & 1  \\
 1 & 1 \\
 1 & -1 \\
 2 & 2 
\end{matrix} & \rvline &
 
\begin{matrix}
-1 \\
0 \\
2  \\
-1
\end{matrix} & \rvline &

\begin{matrix}
  1 & 1 \\
  3 & 3 \\
  1  & 1 \\
  -1 & 0
\end{matrix} 
 \end{pmatrix} \rightarrow
\begin{pmatrix}

\begin{matrix}
 1 & 0  \\
 1 &  0 \\
 -1 & 0 \\
  0 & 1
\end{matrix} & \rvline &

\begin{matrix}
 0 & 1  \\
 1 & 1 \\
 1 & -1 \\
 2 & 2 
\end{matrix} & \rvline &
 
\begin{matrix}
-1 \\
0 \\
2  \\
-1
\end{matrix} & \rvline &

\begin{matrix}
  1 & 1 \\
  3 & 3 \\
  1  & 1 \\
  -1 & 0
\end{matrix} 
 \end{pmatrix} \rightarrow
\begin{pmatrix}

\begin{matrix}
 1 & 0  \\
 0 &  0 \\
 0 & 0 \\
  0 & 1
\end{matrix} & \rvline &

\begin{matrix}
 0 & 1  \\
 1 & 0 \\
 1 & 0 \\
 2 & 2 
\end{matrix} & \rvline &
 
\begin{matrix}
-1 \\
1 \\
1  \\
-1
\end{matrix} & \rvline &

\begin{matrix}
  1 & 1 \\
  2 & 2 \\
  2  & 2 \\
  -1 & 0
\end{matrix} 
 \end{pmatrix} \rightarrow
$$
$$
\begin{pmatrix}

\begin{matrix}
 1 & 0  \\
 0 &  0 \\
 0 & 0 \\
  0 & 1
\end{matrix} & \rvline &

\begin{matrix}
 0 & 1  \\
 1 & 0 \\
 0 & 0 \\
 2 & 2 
\end{matrix} & \rvline &
 
\begin{matrix}
-1 \\
1 \\
0  \\
-1
\end{matrix} & \rvline &

\begin{matrix}
  1 & 1 \\
  2 & 2 \\
  0  & 0 \\
  -1 & 0
\end{matrix} 
 \end{pmatrix}  \rightarrow
\begin{pmatrix}

\begin{matrix}
 1 & 0  \\
 0 &  0 \\
 0 & 0 \\
  0 & 1
\end{matrix} & \rvline &

\begin{matrix}
 0 & 1  \\
 1 & 0 \\
 0 & 0 \\
 2 & 2 
\end{matrix} & \rvline &
 
\begin{matrix}
1 \\
-1 \\
0  \\
1
\end{matrix} & \rvline &

\begin{matrix}
  1 & 1 \\
  2 & 2 \\
  0  & 0 \\
  -1 & 0
\end{matrix} 
 \end{pmatrix} \rightarrow
\begin{pmatrix}
\begin{matrix}
 1 & 0  \\
 0 &  0 \\
 0 & 0 \\
  0 & 1
\end{matrix} & \rvline &

\begin{matrix}
 0 & 1  \\
 1 & 0 \\
 0 & 0 \\
 2 & 2 
\end{matrix} & \rvline &
 
\begin{matrix}
1 \\
-1 \\
0  \\
1
\end{matrix} & \rvline &

\begin{matrix}
  1 & 0 \\
  2 & 0 \\
  0  & 0 \\
  0 & 1
\end{matrix} 
 \end{pmatrix} 
$$
\end{scriptsize}

More explicitly, the steps are
\begin{enumerate}
\item We find the $\pi^{(1)}$-CREF of the first block of $A$ where $\pi^{(1)} = 1234$. We see that $I_1 = 14$.
\item We use the pivot in row 1 to cancel rows 2 and 3. 
\item The second block of $A$ is already in $\pi^{(2)}$-CREF where $\pi^{(2)} = 2314$. 
We see that $I_2 = 12$.
\item We use the pivot in row 2 to cancel row 3. 
\item We find the $\pi^{(3)}$-CREF of the third block of $A$ where $\pi^{(3)} = 3124$. We see that $I_3 = 1$.
\item We find the $\pi^{(4)}$-CREF of the fourth block of $A$ where $\pi^{(4)} = 3124$. We see that $I_4 = 14$.
\end{enumerate}
The output set sequence $\III = (I_1, I_2, I_3, I_4)$ is $(14, 12, 1, 14)$. 
The unipotent subeset $U({\III})$ and the matrix $u \in U(\III)$ are
\begin{scriptsize}
\begin{equation*}
U({\III}) = \begin{pmatrix}
1 & 0 & 0 & 0 \\
\star & 1 & 0 & 0 \\
\star & \star & 1 & 0 \\
0 & 0 & 0 & 1 
\end{pmatrix} \quad
u = \begin{pmatrix}
1 & 0 & 0 & 0 \\
-1 & 1 & 0 & 0 \\
1 & 1 & 1 & 0 \\
0 & 0 & 0 & 1 
\end{pmatrix}.
\end{equation*}
\end{scriptsize}
\end{example}

Observe that once the mixed reduction algorithm processes the $i^{th}$ block of a matrix, that block
remains unchanged for the remainder of the algorithm: any row operations used for later blocks 
will involve two rows of block $i$, the higher of which is a row of zeroes.
The following linear algebra result is the foundation of our nonstandard affine paving of
$Gr(\alpha,k)$.

\begin{proposition}
\label{mixed-reduction}
Let $A \in \VVV_{\alpha}$.  There exists a unique set sequence $\III$ in $[k]$ of type $\alpha$,
a unique matrix $u \in U(\III)$, a unique matrix $B$ fitting the pattern of $\III$, and a unique Levi element
$g \in L_{\alpha}$ such that $A = uBg$.
We have $A \in \UUU_{\alpha}$ if and only if the set sequence $\III$ covers $[k]$.
\end{proposition}

\begin{proof}
Existence comes from applying the mixed reduction algorithm to the matrix $A$.
Uniqueness of $g$ comes from the fact that the blocks $A = (A_1 \mid \cdots \mid A_r)$ of $A$ 
have full rank. Uniqueness of $u, B,$ and $\III$ follow from uniqueness of the $\pi$-CREF for any permutation
$\pi \in \symm_k$.  The last sentence is true because $B$ is nonsingular
if and only if $\III$ covers $[k]$.
\end{proof}

\subsection{Geometric consequences of mixed reduction}
Given a set sequence $\III$, 
we define two subsets $C_{\III}$ and $D_{\III}$ of the Grassmann product as follows.

\begin{defn}  Let $\alpha \in [k]^r$ and
let $\III$ be a set sequence in $[k]$ of type $\alpha$.
Define two subsets $\widehat{D_{\III}} \subseteq \widehat{C_{\III}} \subseteq\VVV_{\alpha}$ by
\begin{align}
\widehat{C_{\III}} &:= \{ u B \,:\, u \in U(\III) \text{ and $B$ fits the pattern of $\III$} \}. \\
\widehat{D_{\III}} &:= \{ B \,:\, \text{ $B$ fits the pattern of $\III$} \}.
\end{align}
Define $D_{\III} \subseteq C_{\III} \subseteq Gr(\alpha,k)$ by
\begin{align}
C_{\III}  &:= \text{image of $\widehat{C_{\III}}$ in $Gr(\alpha,k)$.} \\
D_{\III}  &:= \text{image of $\widehat{D_{\III}}$ in $Gr(\alpha,k)$.}
\end{align}
\end{defn}

The sets $C_{\III}$ and $D_{\III}$ are related by $C_{\III} = U(\III)  D_{\III}$.
A more precise relationship is as follows.

\begin{lemma}
\label{c-are-cells}
Let $\alpha \in [k]^r$ and let $\III = (I_1, \dots, I_r)$ be a set sequence in $[k]$ of type $\alpha$.  
\begin{enumerate}
\item  The subset $D_{\III} \subseteq Gr(\alpha,k)$ is isomorphic to affine space of 
dimension equal to the number of $\star$'s in the pattern matrix $\PM(\III)$.
\item Matrix multiplication gives an isomorphism of varieties 
$\widehat{C_{\III}} L_{\alpha} \cong U(\III) \times \widehat{D_{\III}} L_{\alpha}$.
\item  Matrix multiplication induces an isomorphism of varieties $C_{\III} \cong U(\III) \times D_{\III}$.
\item The subset $C_{\III} \subseteq Gr(\alpha,k)$ is isomorphic to affine space of dimension
equal to the number of $\star$'s in the pattern matrix $\PM(\III)$ plus the
dimension of the affine space $U(\III)$.
\end{enumerate}
\end{lemma}

\begin{proof}
(1) It is clear that $\widehat{D_{\III}}$ is a copy of affine space of the desired dimension. 
By Proposition~\ref{mixed-reduction}, the canonical projection
$\widehat{D_{\III}} \rightarrow D_{\III}$ is a regular bijection. We need only check that the 
inverse function $D_{\III} \rightarrow \widehat{D_{\III}}$ is regular.

The regularity of $D_{\III} \rightarrow \widehat{D_{\III}}$ is a consequence of 
Lemma~\ref{minor-recovery}. Let $\pi^{(1)}, \dots, \pi^{(r)} \in \symm_k$ be the permutation sequence
corresponding to $\III$ and let $A = (A_1 \mid \cdots \mid A_r) \in \widehat{D_{\III}}$.
For each $1 \leq t \leq r$, the matrix $A_t$ is in $\pi^{(t)}$-CREF with pivot row set $I_t$.
The rows of $A$ in $I_t$ are determined by $\pi^{(t)}$.  Lemma~\ref{minor-recovery} says that 
the other entries of $A$ are images under regular functions on $D_{\III}$
of the form $\pm \Delta_{J}^{(t)}/\Delta_{I_t}^{(t)}$, where $|J| = |I_t|$, the sets $J$ and $I_t$ differ in
a single element, 
the sign is deterministic,
and the superscript indicates that we are taking Pl\"ucker coordinates in the $t^{th}$
factor of $D_{\III}$.

(2) Matrix multiplication gives a regular map 
$\mu: U(\III) \times \widehat{D_{\III}} L_{\alpha} \rightarrow \widehat{C_{\III}} L_{\alpha}$.
Proposition~\ref{mixed-reduction} tells us that $\mu$ is a bijection. We need only show that the 
inverse map $\mu^{-1}:  \widehat{C_{\III}} L_{\alpha} \rightarrow U(\III) \times \widehat{D_{\III}} L_{\alpha}$ 
is regular.

Let $u \in U(\III)$ and let $A = (A_1 \mid \cdots \mid A_r) \in \widehat{D_{\III}} L_{\alpha}$ so that
$uA = B = (B_1 \mid \cdots \mid B_r) \in \widehat{C_{\III}} L_{\alpha}$.
We need to show that the entries of $u$ and $A$ are regular functions of the entries of $B$.

Consider applying mixed reduction to the matrix $B$ to get a matrix pair $u' \in U(\III)$ and
$A' \in \widehat{D_{\III}}$.  To process the first block, we find the unique 
$g_1 \in GL_{\alpha_1}$ such that $B_1 g_1$ is in $\pi^{(1)}$-CREF.  The matrix $g_1$ is a regular 
function of the entries of $B_1$ (the rows of $g_1^{-1}$ are the maximal submatrix of $B_1$ with
row set $I_1$). Next, we fill in the nonzero entries in columns $I_1$ of $u'$ 
with entries of $B_1 g_1$ (which are regular functions of $B$) and apply the corresponding 
row operations to $B$ to get a matrix $B' = (A'_1 \mid B'_2 \mid \cdots \mid B'_r)$ (whose entries 
are regular functions of  $B$) and whose first block coincides with the first block of $A'$.
Iterating, we get a matrix $g = g_1 \oplus \cdots \oplus g_t$ in $L_{\alpha}$ such that $B = u'A'g$, 
where each of $u',A,$ and $g$ are regular functions 
of  $B$.  

Since $B = u'A'g = uA$ and $A \in \widehat{D_{\III}} L_{\alpha}$, 
Proposition~\ref{mixed-reduction} tells us that 
$u = u'$ and $A'g = A$.  It follows that the entries of $u$ are regular functions of $B$.
Since the entries of both $A'$ and $g$ are regular functions of $B$, so are the entries of $A = A'g$.

(3) and (4) both follow immediately from (2).
\end{proof}

The cells $C_{\III}$ of Lemma~\ref{c-are-cells} will induce our affine paving of $Gr(\alpha,k)$ 
which interacts nicely with the inclusion $X_{\alpha,k} \subseteq Gr(\alpha,k)$.  The 
set-theoretic result in this direction is as follows.

\begin{lemma}
\label{set-theoretic-stratification}  
Let $\alpha \in [k]^r$.
We have a disjoint union decomposition
\begin{equation}
Gr(\alpha,k) = \bigsqcup_{\III} C_{\III},
\end{equation}
where $\III$ ranges over the collection of all set sequences 
$(I_1, \dots, I_r)$ of type $\alpha$ in $[k]$.  For a given
 set sequence $\III$, we have $C_{\III} \subseteq X_{\alpha,k}$ if $\III$ covers $[k]$ and 
 $C_{\III} \cap X_{\alpha,k} = \varnothing$ otherwise. 
\end{lemma}

\begin{proof}
The disjoint union decomposition follows from Proposition~\ref{mixed-reduction}.
Let $\III$ be a set sequence and suppose that $B$ fits the pattern of $\III$.
From the definition of pattern matrices we see that $B \in \UUU_{\alpha}$ if and only if the 
set sequence $\III$ covers $[k]$, which verifies the second sentence of the lemma.
\end{proof}

The final goal of this section is to show that the cells $C_{\III}$ stratifying 
$Gr(\alpha,k)$ fit together nicely.

\begin{theorem}
\label{affine-paving}
Let $\alpha \in [k]^r$ with $|\alpha| = n$.
The cells 
\begin{equation*}
\{ C_{\III} \,:\, \III = (I_1, \dots, I_r) \text{ is a set sequence in $[k]$ of type $\alpha$} \}
\end{equation*}
 induce an affine paving of 
the Grassmann product $Gr(\alpha,k)$. Furthermore, this affine paving 
\begin{equation*}
\varnothing = X_0 \subset X_1 \subset \cdots \subset X_m = Gr(\alpha,k)
\end{equation*}
may be chosen in such a way that 
$X_{\alpha,k} = Gr(\alpha,k) - X_i$ for some $i$.
\end{theorem}

\begin{proof}
By Lemma~\ref{set-theoretic-stratification} we know that the $C_{\III}$ are cells whose disjoint union is
$Gr(\alpha,k)$.  In order to show that they induce an affine paving of
$Gr(\alpha,k)$, it remains to demonstrate that they may be totally ordered
in such a way that all initial cell unions are Zariski closed in $Gr(\alpha,k)$.

Given any $k \times n$ matrix 
$A$ and any position $(i,j) \in [k] \times [n]$, we define a number
\begin{equation}
\rk(i,j,A) := \text{rank of the northwest $i \times j$ submatrix of $A$}.
\end{equation}
If $u \in U \subseteq GL_k$ is any lower triangular $k \times k$ unipotent matrix, we have
\begin{equation}
\rk(i,j,uA) = \rk(i,j,A)
\end{equation}
since we are taking northwest submatrices. Furthermore, for any $1 \leq j \leq n$ we have the growth conditions
\begin{equation}
\rk(i-1,j,A) \leq \rk(i,j,A) \leq \rk(i-1,j,A) + 1, \quad i = 1, 2, \dots, k
\end{equation}
where we adopt the convention $\rk(0,j,A) = 0$.

The function $\rk(i,j,-)$ is  well-defined  on 
$Gr(\alpha,k)$ whenever $j \in \{ \alpha_1, \alpha_1 + \alpha_2, \dots \}$.
We  have the disjoint union decomposition
\begin{equation}
\label{first-grassmann-decomposition}
Gr(\alpha,k) = \bigsqcup_{\rho} \Omega_{\rho},
\end{equation}
where $\rho$ varies over all functions 
$\rho: [k] \times \{ \alpha_1, \alpha_1 + \alpha_2, \dots \} \rightarrow \ZZ_{\geq 0}$
and $\Omega_{\rho}$ is given by
\begin{equation}
\Omega_{\rho} := \left\{ [A]  \in Gr(\alpha,k) \,:\, 
\rk(i,j,A) = \rho(i,j) \text{ for all $1 \leq i \leq k$ and $j \in \{ \alpha_1, \alpha_1 + \alpha_2, \dots \}$} \right\},
\end{equation}
where $A \in \VVV_{\alpha}$ represents the point $[A] \in Gr(\alpha,k)$.
Although the indexing set of the decomposition~\eqref{first-grassmann-decomposition} is infinite,
all but finitely many of the sets $\Omega_{\rho}$ are empty.

If we totally order the sets $\Omega_{\rho}$ in the decomposition~\eqref{first-grassmann-decomposition} in 
an extension of the partial order on functions $\rho$ given by 
$\rho \leq \rho'$ if and only if $\rho(i,j) \leq \rho'(i,j)$ for all $i, j$, then any initial union of the $\Omega_{\rho}$
will be Zariski closed in $Gr(\alpha,k)$.
The idea is to show that, for fixed $\rho$, the set $\Omega_{\rho}$ is isomorphic to a product of 
affine spaces and  Grassmannians, and that appropriate cells $C_{\III}$ cut out a Schubert-type 
decomposition of this product.
To see what the $\Omega_{\rho}$ look like, 
we give a more combinatorial version of the decomposition~\eqref{first-grassmann-decomposition}.

Let $\rho: [k] \times \{ \alpha_1, \alpha_1 + \alpha_2, \dots \} \rightarrow \ZZ_{\geq 0}$ be  such that 
$\Omega_{\rho} \neq \varnothing$.  We have the slow growth condition
\begin{equation}
\label{slow-growth-row}
\rho(i-1,j) \leq \rho(i,j) \leq \rho(i-1,j) + 1,
\end{equation}
where we set $\rho(0,j) := 0$. This motivates the {\em jump sets} $J_t(\rho) \subseteq [k]$ given by
\begin{equation}
\label{jump-set-definition}
J_t(\rho) := \{ 1 \leq i \leq k \,:\, \rho(i-1, \alpha_1 + \cdots + \alpha_t) < \rho(i, \alpha_1 + \cdots + \alpha_t) \}, \quad
1 \leq t \leq r.
\end{equation}
The function $\rho$ can be recovered from the jump sets $(J_1(\rho), \dots, J_r(\rho))$.
Since $\Omega_{\rho} \neq \varnothing$, these jump sets satisfy the containments
\begin{equation}
\label{jump-set-containment}
J_1(\rho) \subseteq J_2(\rho) \subseteq \cdots \subseteq J_r(\rho)
\end{equation}
as well as the size conditions
\begin{equation}
\label{jump-set-size}
\max(\alpha_i, |J_{i-1}(\rho)|) \leq |J_i(\rho)| \leq |J_{i-1}(\rho)| + \alpha_i,
\end{equation}
where we set $J_0(\rho) := \varnothing$.
We say that a length $r$ sequence $\JJJ = (J_1, \dots, J_r)$ of subsets of $[k]$ 
is a {\em jump sequence} if it satisfies the conditions
\eqref{jump-set-containment} and \eqref{jump-set-size}.  If we let $\Omega_{\JJJ}$ be the family of points
in $Gr(\alpha,k)$ whose rank function $\rho$ has jump set sequence $\JJJ$, we have the following finite and irredundant
version of the decomposition~\ref{first-grassmann-decomposition}:
\begin{equation}
\label{second-grassmann-decomposition}
Gr(\alpha,k) = \bigsqcup_{\JJJ} \Omega_{\JJJ},
\end{equation}
where the union is over all jump sequences $\JJJ$.  

Fix a jump sequence $\JJJ = (J_1, \dots, J_r)$ for the remainder 
of the proof and consider the piece $\Omega_{\JJJ}$
of the decomposition~\eqref{second-grassmann-decomposition}.
Let $\III = (I_1, \dots, I_r)$ be a set sequence in $[k]$ of type $\alpha$ and consider the associated cell $C_{\III}$.
These two subsets of $Gr(\alpha,k)$ are related by the following dichotomy.
\begin{quote}
{\em If $I_1 \cup \cdots \cup I_i =  J_i$ for all $i = 1, 2, \dots, r$ then
$C_{\III} \subseteq \Omega_{\JJJ}$.  Otherwise $C_{\III} \cap \Omega_{\JJJ} = \varnothing$.}
\end{quote}
This dichotomy comes from the fact that if $A$ is any matrix which fits the pattern of $\III$, then
\begin{equation}
\rk(i, \alpha_1 + \cdots + \alpha_t, A) = \text{number of $1$'s northwest of $(i, \alpha_1 + \cdots + \alpha_t)$
in $\PM(\III)$.}
\end{equation}
We have the disjoint union decomposition
\begin{equation}
\label{omega-decomposition}
\Omega_{\JJJ} = \bigsqcup_{\III} C_{\III},
\end{equation}
where $\III$ ranges over all set sequences $\III = (I_1, \dots, I_r)$ in $[k]$ of type $\alpha$ such that
$I_1 \cup \cdots \cup I_i = J_i$ for $i = 1, 2, \dots, r$; call such set sequences
{\em $\JJJ$-compatible}.

By Observation~\ref{jump-observation}, for any two $\JJJ$-compatible set sequences $\III, \III'$ 
we have $\pi^{(t)}_{\III} = \pi^{(t)}_{\III'}$ for all $1 \leq t \leq r$ and $U(\III) = U(\III')$.  We therefore let
$\pi^{(1)}, \dots, \pi^{(t)} \in \symm_k$ and $U(\JJJ)$ be the permutation sequence
and unipotent matrix set associated to any $\JJJ$-compatible set sequence $\III$.

We have a bijection induced by matrix multiplication
\begin{equation}
\label{second-omega-decomposition}
\Omega_{\JJJ} = \bigsqcup_{\text{$\III$ is $\JJJ$-compatible}} C_{\III} \cong 
U(\JJJ) \times \bigsqcup_{\text{$\III$ is $\JJJ$-compatible}} D_{\III}
\end{equation}
where $C_{\III} \cong U(\JJJ) \times D_{\III}$ by Lemma~\ref{c-are-cells} (3).
An argument similar to that of Lemma~\ref{c-are-cells} shows that the bijection
\eqref{second-omega-decomposition} is an isomorphism of varieties.

The cells $D_{\III}$ for $\JJJ$-compatible set sequences $\III$ induce an affine paving
of the disjoint union $ \bigsqcup_{\text{$\III$ is $\JJJ$-compatible}} D_{\III}$ forming the second
factor of Equation~\eqref{second-omega-decomposition}. 
This affine paving is  best understood with an example.
Suppose $k = 4$, $\alpha = (2,2,2)$, and $\JJJ = (13, 134, 134)$.
There are six $\JJJ$-compatible set sequences $\III = (I_1, I_2, I_3)$ of type $\alpha$.
The corresponding cells $D_{\III}$ are displayed below.
\begin{center}
\begin{scriptsize}
\begin{tabular}{c c c}
$\begin{pmatrix}
\begin{matrix}
1 & 0 \\ 0 & 0 \\ 0 & 1 \\ 0 & 0
\end{matrix} & \rvline &
\begin{matrix} \star & 0 \\ 0 & 0 \\ 0 & 1 \\ 1 & 0 \end{matrix} & \rvline &
\begin{matrix} 0 & 0 \\ 0 & 0 \\ 1 & 0 \\ 0 & 1 \end{matrix}
\end{pmatrix}$ & 
$\begin{pmatrix}
\begin{matrix}
1 & 0 \\ 0 & 0 \\ 0 & 1 \\ 0 & 0
\end{matrix} & \rvline &
\begin{matrix} 0 & 1 \\ 0 & 0 \\ \star & \star \\ 1 & 0 \end{matrix} & \rvline &
\begin{matrix} 0 & 0 \\ 0 & 0 \\ 1 & 0 \\ 0 & 1 \end{matrix}
\end{pmatrix}$ & 
$\begin{pmatrix}
\begin{matrix}
1 & 0 \\ 0 & 0 \\ 0 & 1 \\ 0 & 0
\end{matrix} & \rvline &
\begin{matrix} \star & 0 \\ 0 & 0 \\ 0 & 1 \\ 1 & 0 \end{matrix} & \rvline &
\begin{matrix} 1 & 0 \\ 0 & 0 \\ \star & 0 \\ 0 & 1 \end{matrix}
\end{pmatrix}$ \\ \vspace{0.04in} \\
$\III = (13, 34, 34)$ & $\III = (13, 14, 34)$ & $\III = (13, 34, 14)$
\end{tabular}
\end{scriptsize}
\end{center}

\begin{center}
\begin{scriptsize}
\begin{tabular}{c c c}
$\begin{pmatrix}
\begin{matrix}
1 & 0 \\ 0 & 0 \\ 0 & 1 \\ 0 & 0
\end{matrix} & \rvline &
\begin{matrix} 0 & 1 \\ 0 & 0 \\ \star & \star \\ 1 & 0 \end{matrix} & \rvline &
\begin{matrix} 1 & 0 \\ 0 & 0 \\ \star & 0 \\ 0 & 1 \end{matrix}
\end{pmatrix}$ & 
$\begin{pmatrix}
\begin{matrix}
1 & 0 \\ 0 & 0 \\ 0 & 1 \\ 0 & 0
\end{matrix} & \rvline &
\begin{matrix} \star & 0 \\ 0 & 0 \\ 0 & 1 \\ 1 & 0 \end{matrix} & \rvline &
\begin{matrix} 1 & 0 \\ 0 & 0 \\ 0 & 1 \\ \star & \star \end{matrix}
\end{pmatrix}$ & 
$\begin{pmatrix}
\begin{matrix}
1 & 0 \\ 0 & 0 \\ 0 & 1 \\ 0 & 0
\end{matrix} & \rvline &
\begin{matrix} 0 & 1 \\ 0 & 0 \\ \star & \star \\ 1 & 0 \end{matrix} & \rvline &
\begin{matrix} 1 & 0 \\ 0 & 0 \\ 0 & 1 \\ \star & \star \end{matrix}
\end{pmatrix}$ \\ \vspace{0.04in} \\
$\III = (13, 14, 14)$ & $\III = (13, 34, 13)$ & $\III = (13, 14, 13)$
\end{tabular}
\end{scriptsize}
\end{center}

The product of the affine paving of 
$ \bigsqcup_{\text{$\III$ is $\JJJ$-compatible}} D_{\III}$ shown above with the cell $\Omega_{\JJJ}$
gives the desired paving of $\Omega_{\JJJ}$ by the cells $C_{\III}$
by means of \eqref{second-omega-decomposition}.
\end{proof}

\section{Cohomology presentation}
\label{Cohomology}

\subsection{$\ZZ$-modules}
We aim to show that $H^{\bullet}(X_{\alpha,k}; \ZZ)$ is the $\symm_{\alpha}$-invariant 
subring $(R_{\alpha,k})^{\symm_{\alpha}}$ of $R_{\alpha,k}$.
Theorem~\ref{s-structure-theorem} gives information about the 
rational version $S_{\alpha,k} = \QQ \otimes_{\ZZ} R_{\alpha,k}$ of $R_{\alpha,k}$.
We will need to relate vector spaces over $\QQ$
and modules over $\ZZ$.  We collect several basic facts in this direction here.
Recall that a $\ZZ$-module $M$ is {\em free of rank $r$} if
$M$ has an $r$-element $\ZZ$-basis.

\begin{lemma}
\label{basic-module-facts}  
\begin{enumerate}
\item Let $M$ be a free $\ZZ$-module of finite rank, so that $\QQ \otimes_{\ZZ} M$ is a $\QQ$-vector space.
We have $\dim(\QQ \otimes_{\ZZ} M) = \mathrm{rank}(M)$.
\item Let $M$ be a free $\ZZ$-module of finite rank and let $N \subseteq M$ be a submodule. Then $N$ is also a free
$\ZZ$-module.
\item Let $f: M \twoheadrightarrow N$ be a surjective homomorphism between free $\ZZ$-modules of finite rank.
If $\mathrm{rank}(M) = \mathrm{rank}(N)$ then $f$ is an isomorphism.
\end{enumerate}
\end{lemma}

Let $X$ be a complex algebraic variety and let $V$ be a vector bundle over $X$.
Let $X' \rightarrow X$ be the bundle over $X$ whose fiber over $x \in X$ is the space of complete flags in $V_x$.
To present the integral cohomology $H^{\bullet}(X_{\alpha,k}; \ZZ)$, we will need  some 
facts about the embedding $H^{\bullet}(X;\ZZ) \hookrightarrow H^{\bullet}(X';\ZZ)$.

The bundle $X' \rightarrow X$ is a sequence of projective bundles.
Let $V \rightarrow X$ be a vector bundle of rank $r$ 
 and let $p: \mathbb{P}(V) \rightarrow X$ be the associated projective bundle.
 Then the induced map $p^*: H^{\bullet}(X;\ZZ) \rightarrow H^{\bullet}(\mathbb{P}(V);\ZZ)$ is injective.
 In fact, we may present $H^{\bullet}(\mathbb{P}(V); \ZZ)$ as 
 \begin{equation}
 H^{\bullet}(\mathbb{P}(V); \ZZ) = 
 H^{\bullet}(X;\ZZ)[\zeta] / \langle \zeta^r + c_1(V) \zeta^{r-1} + \cdots + c_r(V) \rangle,
 \end{equation}
 where the $c_i(V) \in H^{\bullet}(X;\ZZ)$ are the Chern classes of $V$ and $\zeta \in  H^{2}(\mathbb{P}(V); \ZZ)$
 is the first Chern class of the tautological line bundle $\mathcal{O}(1)$ on $\mathbb{P}(V)$.

 In our setting, the cohomology $H^{\bullet}(X;\ZZ)$ will be a free $\ZZ$-module of finite rank.
 The above discussion motivates the following lemma.

\begin{lemma}
\label{free-module-extension}
Let $A$ be a free $\ZZ$-module of finite rank and let $t$ be a variable. Fix elements $a_1, a_2, \dots, a_r \in A$
and consider the quotient 
\begin{equation*}
B := A[t] / \langle t^r + a_{1} t^{r-1} + \cdots + a_{r-1} t + a_r \rangle.
\end{equation*}
We have the following.
\begin{enumerate}
\item
The ring $B$ is a free $\ZZ$-module.  
\item
The natural map $A \rightarrow B$ is injective so we may view $A$ as a subset of $B$.
\item 
For any integer $N > 0$ and any $b \in B$ such that $N \cdot b \in A$ we have $b \in A$.
\end{enumerate}
\end{lemma}

\begin{proof}
If $\alpha_1, \dots, \alpha_m$ is a $\ZZ$-basis of $A$, it follows that 
$\{ \alpha_i t^j \,:\, 1 \leq i \leq m, \, 0 \leq j \leq r-1\}$
is a $\ZZ$-basis of $B$, which proves (1). Statements (2) and (3) follow from the form of this basis.
\end{proof}

\subsection{The ring $H^{\bullet}(X_{\alpha,k}; \ZZ)$}
As a first application of Lemma~\ref{basic-module-facts}, we have the following.

\begin{lemma}
\label{r-structure-lemma}
Let $\alpha = (\alpha_1, \dots, \alpha_r) \in [k]^r$ with $|\alpha| = n$.  
\begin{enumerate}
\item
The ring $R_{\alpha,k}$ is a free $\ZZ$-module of rank
$|\OP_{\alpha,k}|$.  
\item
The invariant subring $(R_{\alpha,k})^{\symm_{\alpha}}$ is a free $\ZZ$-module of rank
$|\OP_{\alpha,k}|/(\alpha_1! \cdots \alpha_r!)$.
\end{enumerate}
\end{lemma}

\begin{proof}
(1)  We show that the collection $\MMM_{\alpha,k}$ of $\alpha$-nonskip monomials descends
to a $\ZZ$-basis of $R_{\alpha,k}$. 
By Lemma~\ref{cardinality-coincidence} we have $|\MMM_{\alpha,k}| = |\OP_{\alpha,k}|$, so would prove the claim. 
By Theorem~\ref{s-structure-theorem}, we know that these monomials
descend to a $\QQ$-basis of $S_{\alpha,k} = \QQ \otimes_{\ZZ} R_{\alpha,k}$, so it suffices to prove that
these monomials span $R_{\alpha,k}$ over $\ZZ$.

Let $m$ be any monomial in $\ZZ[\xx_n]$. We need to show that $m$ lies in the span of $\MMM_{\alpha,k}$
modulo $I_{\alpha,k}$.  If $m \in \MMM_{\alpha,k}$ this is clear, so assume that $m \not\in \MMM_{\alpha,k}$.
 The last sentence of 
Proposition~\ref{groebner-proposition} implies that there is a nonzero polynomial $f \in I_{\alpha,k}$ 
(either a Demazure character or a complete homogeneous symmetric polynomial) such that 
$\initial_<(f) \mid m$ with respect to the $\neglex$ term order and that the $\neglex$-leading coefficient
of $f$ equals $1$.
 If we let $m'$ be the quotient
$m' := m/\initial_<(f)$, the $\neglex$-leading term of $m' \cdot f \in I_{\alpha,k}$ is $m$.  
Since $m' \cdot f \equiv 0$
modulo $I_{\alpha,k}$, we get
\begin{equation}
m \cong \text{a $\ZZ$-linear combination of monomials $< m$ in $\neglex$}.
\end{equation}
We are done by induction on the $\neglex$ order.

(2)  By Lemma~\ref{basic-module-facts} (2), we know that $(R_{\alpha,k})^{\symm_{\alpha}}$
is a free $\ZZ$-module.  By Theorem~\ref{s-structure-theorem}, we have
$S_{\alpha,k} \cong \QQ[\OP_{\alpha,k}]$ as $\QQ[\symm_{\alpha}]$-modules. Since the action of 
$\symm_{\alpha}$ on $\OP_{\alpha,k}$ is free, taking 
$\symm_{\alpha}$-invariants and applying Lemma~\ref{basic-module-facts} (1) gives
\begin{equation}
\mathrm{rank}(R_{\alpha,k})^{\symm_{\alpha}} = 
\dim \QQ \otimes_{\ZZ} (R_{\alpha,k})^{\symm_{\alpha}} = \dim S_{\alpha,k}^{\symm_{\alpha}} = 
\dim \QQ[\OP_{\alpha,k}]^{\symm_{\alpha}} = |\OP_{\alpha,k}|/|\symm_{\alpha}|,
\end{equation}
as desired.
\end{proof}

We are ready to prove Theorem~\ref{main-theorem} and present the cohomology of $X_{\alpha,k}$.
Let us recall the statement of the result.

\noindent
{\bf Theorem 1.4.}
{\em Let $\alpha = (\alpha_1, \dots, \alpha_r) \in [k]^r$ be a sequence of positive integers with 
$\alpha_1 + \cdots + \alpha_r = n$.  
The singular cohomology ring $H^{\bullet}(X_{\alpha,k}; \ZZ)$ may be presented as 
\begin{equation}
H^{\bullet}(X_{\alpha,k}; \ZZ) = 
(\ZZ[\xx_n]/I_{\alpha,k})^{\symm_{\alpha}} = (R_{\alpha,k})^{\symm_{\alpha}},
\end{equation}
where $I_{\alpha,k} \subseteq \ZZ[\xx_n]$ is the ideal generated by 
the complete homogeneous
symmetric polynomials $h_d(\xx_n^{(i)})$ where $d > k - \alpha_i$ and $1 \leq i \leq r$
together with the elementary symmetric polynomials $e_n(\xx_n), e_{n-1}(\xx_n), \dots, e_{n-k+1}(\xx_n)$.
Here the variables  $\xx^{(i)}_n$ represent the Chern roots of the tautological vector bundle 
$W_i^* \twoheadrightarrow X_{\alpha,k}$.}

\begin{proof}
Theorem~\ref{affine-paving} implies that the inclusion $\iota: X_{\alpha,k} \hookrightarrow Gr(\alpha,k)$
induces a surjective map 
\begin{equation}
\iota^*: H^{\bullet}(Gr(\alpha,k); \ZZ) \twoheadrightarrow H^{\bullet}(X_{\alpha,k}; \ZZ)
\end{equation}
on cohomology.
It follows that $H^{\bullet}(X_{\alpha,k}; \ZZ)$ is generated by the Chern classes of the bundles $W_1^*, \dots, W_r^*$.
In terms of the variables $x_i$, we get the following.
\begin{quote}
{\em The cohomology ring $H^{\bullet}(X_{\alpha,k}; \ZZ)$ is generated by the elementary symmetric polynomials
in the partial variable sets
\begin{equation}
e_d(\xx_n^{(i)}), \quad
1 \leq i \leq r, \, \, 1 \leq d \leq \alpha_i.
\end{equation}
Pushing forward relations under $\iota^*$, we have
\begin{equation}
h_d(\xx_n^{(i)}) = 0, \quad
1 \leq i \leq r, \, \, k - \alpha_i + 1 \leq d \leq k
\end{equation}
 in $H^{\bullet}(X_{\alpha,k}; \ZZ)$.}
\end{quote}

The Whitney Sum Formula provides  additional relations  in $H^{\bullet}(X_{\alpha,k}; \ZZ)$.
The addition map $(w_1, \dots, w_r) \mapsto w_1 + \cdots + w_r$ gives a vector 
bundle surjection
\begin{equation}
W_1 \oplus \cdots \oplus W_r \twoheadrightarrow \CC^k,
\end{equation}
where $\CC^k$ is the trivial rank $k$ vector bundle over $X_{\alpha,k}$.  This dualizes to give an injection
\begin{equation}
(\CC^k)^* \hookrightarrow W_1^* \oplus \cdots \oplus W_r^*,
\end{equation}
so that the direct sum $W_1^* \oplus \cdots \oplus W_r^*$ has a trivial rank $k$ subbundle.
If we let $U$ be the corresponding rank $n-k$ quotient of $W_1^* \oplus \cdots \oplus W_r^*$,
the Whitney Sum Formula tells us that 
\begin{equation}
c_{\bullet}(U) = c_{\bullet}(W_1^*) \cdots c_{\bullet}(W_r^*)
\end{equation}
where 
\begin{equation}
c_{\bullet}(W_i^*) = (1 + x_{\alpha_1 + \cdots + \alpha_{i-1} + 1} t) 
 (1 + x_{\alpha_1 + \cdots + \alpha_{i-1} + 2} t) \cdots 
  (1 + x_{\alpha_1 + \cdots + \alpha_{i-1} + \alpha_i} t).
\end{equation}
Since $c_{\bullet}(U)$ is a polynomial in $t$ of degree $\leq n-k$, we have
$e_d(\xx_n) = 0$ in $H^{\bullet}(X_{\alpha,k}; \ZZ)$ 
for each $d > n-k$.

Let $p: X'_{\alpha,k} \rightarrow X_{\alpha,k}$ be the bundle whose fiber over a point 
$(W_1, \dots, W_r) \in X_{\alpha,k}$ is the product space 
$\mathcal{F \ell}(W_1) \times \cdots \times \mathcal{F \ell}(W_r)$ of complete flags in $W_1, \dots, W_r$.
The Chern roots $x_1, \dots, x_n$ are elements of $H^{\bullet}(X'_{\alpha,k}; \ZZ)$.
Any $\symm_{\alpha}$-invariant polynomial in these Chern roots lies in $H^{\bullet}(X_{\alpha,k}; \ZZ)$.
The map $p$ induces an inclusion $p^*: H^{\bullet}(X_{\alpha,k}; \ZZ) \hookrightarrow H^{\bullet}(X'_{\alpha},\ZZ)$.
The second sentence in italics and the paragraph above imply that we have a map 
of $\ZZ$-algebras
\begin{equation}
\widehat{\varphi}: (R_{\alpha,k})^{\symm_{\alpha}} = (\ZZ[\xx_n]/I_{\alpha,k})^{\symm_{\alpha}} 
\rightarrow H^{\bullet}(X_{\alpha,k}'; \ZZ).
\end{equation}

We claim that the image of $\widehat{\varphi}$ in fact lies in the subring $H^{\bullet}(X_{\alpha,k};\ZZ)$ of 
$H^{\bullet}(X'_{\alpha,k};\ZZ)$.
If we were using rational coefficients, this would be shown as follows.
Let $\alpha! := \alpha_1! \cdots \alpha_r!$ be the size of the group $\symm_{\alpha}$.
Given any $f \in (R_{\alpha,k})^{\symm_{\alpha}}$ we could average 
$\widehat{\varphi}(f)$ over $\symm_{\alpha}$
to obtain
\begin{equation}
\widehat{\varphi}(f) =  
\widehat{\varphi} \left( \frac{1}{\alpha!} \sum_{\pi \in \symm_{\alpha}} \pi \cdot f \right) = 
\frac{1}{\alpha!} 
\sum_{\pi \in \symm_{\alpha}} \pi \cdot \widehat{\varphi}(f),
\end{equation}
where the first equality would use the fact that $f$ is $\symm_{\alpha}$-invariant.
We see that $\widehat{\varphi}(f)$ would also be 
also $\symm_{\alpha}$-invariant and would lie in the cohomology of $X_{\alpha,k}$.

Since we are using integer coefficients, we cannot divide by $\alpha!$ and the argument in the above 
paragraph does not work.
However, this argument can be salvaged.
For $f \in (R_{\alpha,k})^{\symm_{\alpha}}$
the image 
\begin{equation} \alpha! \cdot \widehat{\varphi}(f) =
\widehat{\varphi}(\alpha! \cdot f) = 
\widehat{\varphi} \left( \sum_{\pi \in \symm_{\alpha}} \pi \cdot f \right) = 
\sum_{\pi \in \symm_{\alpha}} \pi \cdot \widehat{\varphi}(f),
\end{equation}
of $\alpha! \cdot f$ under $\widehat{\varphi}$ is $\symm_{\alpha}$-invariant and so lies in 
$H^{\bullet}(X_{\alpha,k}; \ZZ)$.
We need to show that $\widehat{\varphi}(f)$ itself lies in $H^{\bullet}(X_{\alpha,k}; \ZZ)$.
Since $p: X'_{\alpha,k} \rightarrow X_{\alpha,k}$ is a sequence of projective bundles and 
$H^{\bullet}(X_{\alpha,k}; \ZZ)$ is a free $\ZZ$-module of finite rank, the ring extension 
$p^*: H^{\bullet}(X_{\alpha,k};\ZZ) \hookrightarrow H^{\bullet}(X'_{\alpha,k};\ZZ)$
is a sequence of extensions of the form addressed in Lemma~\ref{free-module-extension}. 
By Lemma~\ref{free-module-extension} (3) and induction, we have 
$\widehat{\varphi}(f) \in H^{\bullet}(X_{\alpha,k};\ZZ)$.

Since the image of $\widehat{\varphi}$ lies in $H^{\bullet}(X_{\alpha,k};\ZZ)$ we may restrict the codomain
of $\widehat{\varphi}$ to get a map
of $\ZZ$-algebras
\begin{equation}
\varphi: (R_{\alpha,k})^{\symm_{\alpha}} = (\ZZ[\xx_n]/I_{\alpha,k})^{\symm_{\alpha}} 
\rightarrow H^{\bullet}(X_{\alpha,k}; \ZZ).
\end{equation}
The first sentence in italics above implies that $\varphi$ is a surjection.
By Lemma~\ref{r-structure-lemma} and Theorem~\ref{affine-paving}
(as well as Lemma~\ref{set-theoretic-stratification}), 
we know that the domain
and codomain of $\varphi$ are free of rank $|\OP_{\alpha,k}|/\alpha!$.
By Lemma~\ref{basic-module-facts} (3) we conclude that $\varphi$ is an isomorphism.
\end{proof}

Since the cohomology of $X_{\alpha,k}$ is concentrated in even degree,
Theorem~\ref{main-theorem} and the Universal Coefficient Theorem
imply that the rational cohomology of $X_{\alpha,k}$ is given by
\begin{equation}
\label{rational-cohomology-presentation}
H^{\bullet}(X_{\alpha,k}; \QQ) = \QQ \otimes_{\ZZ} (R_{\alpha,k})^{\symm_{\alpha}} = (S_{\alpha,k})^{\symm_{\alpha}}.
\end{equation}

The description of $H^{\bullet}(X_{\alpha,k}; \ZZ) = (R^{\alpha})^{\symm_{\alpha}}$ in terms of Chern roots
implies the following result.

\begin{corollary}
\label{symmetric-action-corollary}
Let $\alpha = (\alpha_1, \dots, \alpha_r) \in [k]^r$ with $|\alpha| = n$.
Let $\pi \in \symm_r$ and
set $\pi.\alpha := (\alpha_{\pi(1)}, \dots, \alpha_{\pi(r)})$.
We have the following pair of commutative diagrams.
\begin{center}
\begin{tikzpicture}
\node at (0,0) (A) {$(R_{\alpha,k})^{\symm_{\alpha}}$};

\node at (3,0) (B) {$(R_{\pi.\alpha})^{\symm_{\pi.\alpha}}$};

\node at (0,2) (C) {$H^{\bullet}(X_{\alpha,k}; \ZZ)$};

\node at (3,2) (D) {$H^{\bullet}(X_{\pi.\alpha}; \ZZ)$};

\node at (7,0) (E) {$(S_{\alpha,k})^{\symm_{\alpha}}$};

\node at (10,0) (F) {$(S_{\pi.\alpha})^{\symm_{\pi.\alpha}}$};

\node at (7,2) (G) {$H^{\bullet}(X_{\alpha,k}; \QQ)$};

\node at (10,2) (H) {$H^{\bullet}(X_{\pi.\alpha}; \QQ)$};

\draw [->] (A) -- (B);

\draw [->] (C) -- (D);

\draw [->] (E) -- (F);

\draw [->] (G) -- (H);

\draw [double, double distance = 0.04cm] (A) -- (C);

\draw [double, double distance = 0.04cm] (B) -- (D);

\draw [double, double distance = 0.04cm] (E) -- (G);

\draw [double, double distance = 0.04cm] (F) -- (H);
\end{tikzpicture}
\end{center}
The vertical equalities come from Theorem~\ref{main-theorem}
(and Equation~\eqref{rational-cohomology-presentation}).
The top horizontal arrows are induced from the action on subspaces 
$(W_1, \dots, W_r) \mapsto (W_{\pi(1)}, \dots, W_{\pi(r)})$.
The bottom horizontal arrows are induced by the action on batches of variables 
$(\xx_n^{(1)}, \dots, \xx_n^{(r)}) \mapsto (\xx_n^{(\pi(1))}, \dots, \xx_n^{(\pi(r))})$.
\end{corollary}

When the dimension vector $\alpha = (d, \dots, d) \in [k]^r$ is constant, we have a combinatorial
model for the ungraded representation $H^{\bullet}(X_{\alpha,k}; \QQ)$ of 
$\symm_r$.

\begin{corollary}
\label{ungraded-combinatorial-model}
Let $\alpha = (d, \dots, d) \in [k]^r$ be a constant dimension vector with $|\alpha| = rd = n$ 
so that $H^{\bullet}(X_{\alpha,k}; \QQ)$ is a $\symm_r$-module.
Let $\CCC_{\alpha}$ be the family of set sequences $\III = (I_1, \dots, I_r)$ in $[k]$
of type $\alpha$ which cover $[k]$.  

The symmetric group $\symm_r$ acts on $\CCC_{\alpha}$ by subscript permutation, viz.
$\pi.(I_1, \dots, I_r) := (I_{\pi(1)}, \dots, I_{\pi(r)})$.
We have an isomorphism of 
ungraded $\symm_r$-modules
\begin{equation}
H^{\bullet}(X_{\alpha,k}; \QQ) \cong \QQ[\CCC_{\alpha}].
\end{equation}
\end{corollary}

\begin{proof}
The {\em wreath product} $\symm_d \wr \symm_r$ is the subgroup of $\symm_n$ given by
the semidirect product
\begin{equation*} \symm_d \wr \symm_r :=
\symm_r \rtimes (\symm_d \times \cdots \times \symm_d) = \symm_r \rtimes \symm_{\alpha},
\end{equation*}
where there are $r$ factors of $\symm_d$.  The symmetric group $\symm_r$ acts on the $r$-fold
product $\symm_{\alpha}$ by factor permutation.
We have a natural embedding 
\begin{equation*}
\symm_{\alpha} = 1 \rtimes \symm_{\alpha} \subseteq \symm_r \rtimes \symm_{\alpha} = \symm_d \wr \symm_r.
\end{equation*}

The point locus $Z_{\alpha,k} \subseteq \QQ^n$ of 
Definition~\ref{point-locus-definition} is closed under the action of $\symm_d \wr \symm_r$.
Theorem~\ref{s-structure-theorem} (and its proof) give an isomorphism of ungraded
$\symm_d \wr \symm_r$-modules 
$S_{\alpha,k} \cong \QQ[\OP_{\alpha,k}]$. Taking $\symm_{\alpha}$-invariants and looking 
and the residual action of $\symm_r$ gives the corollary.
\end{proof}

\section{Open Problems}
\label{Open}

For any $\alpha \in [k]^r$ with $|\alpha| = n$, 
Theorem~\ref{main-theorem} gives a  geometric model for the ring
of $\symm_{\alpha}$-invariants  $(R_{\alpha,k})^{\symm_{\alpha}}$.  It is natural to ask for a 
geometric model for $R_{\alpha,k}$ itself.  One candidate is as follows.

\begin{defn}
\label{line-configuration-definition}
Let $\alpha = (\alpha_1, \dots, \alpha_r) \in [k]^r$ with $|\alpha| = n$.  Let $Y_{\alpha,k}$ be the moduli
space of $n$-tuples $(\ell_1, \dots, \ell_n)$ of $1$-dimensional subspaces $\ell_i \subseteq \CC^k$ such that
\begin{itemize}
\item we have $\ell_1 + \cdots + \ell_n = \CC^k$ and
\item for each $1 \leq i \leq r$, the subspace 
$\ell_{\alpha_1 + \cdots + \alpha_{i-1} + 1} + \cdots + \ell_{\alpha_1 + \cdots + \alpha_{i-1} + \alpha_i}$
of $\CC^k$
spanned by the $i^{th}$ batch of lines is $\alpha_i$-dimensional.
\end{itemize}
\end{defn}

Wilson and the author \cite{RW} studied $Y_{\alpha,k}$ in the special case $\alpha = (m,1, \dots, 1)$.

\begin{conjecture}
\label{y-cohomology-conjecture}
We have the presentation $H^{\bullet}(Y_{\alpha,k}; \ZZ) = R_{\alpha,k}$ where the variable $x_i$ represents
the Chern class $c_1(\ell_i^*)$ of the line bundle $\ell_i^* \twoheadrightarrow Y_{\alpha,k}$.
\end{conjecture}

Conjecture~\ref{y-cohomology-conjecture} is supported by the following fact about rational cohomology.
There is a natural projection
\begin{equation}
\label{fibration}
\pi: Y_{\alpha,k} \rightarrow X_{\alpha,k}
\end{equation}
which sends the line tuple $(\ell_1, \dots, \ell_n)$ to the subspace tuple $(W_1, \dots, W_r)$ 
where 
\begin{equation*}
W_i = \ell_{\alpha_1 + \cdots + \alpha_{i-1} + 1} + \cdots + \ell_{\alpha_1 + \cdots + \alpha_{i-1} + \alpha_i}
\end{equation*}
is the span of the $i^{th}$ batch of lines.

The map $\pi$ is a fiber bundle with fiber $F$ homotopy equivalent to the product 
\begin{equation}
\mathcal{F \ell}(\alpha_1) \times \cdots \times \mathcal{F \ell}(\alpha_r)
\end{equation}
of flag varieties.
The inclusion $\iota: F \hookrightarrow Y_{\alpha,k}$ endows $H^{\bullet}(F; \QQ)$ 
induces a map $\iota^*: H^{\bullet}(Y_{\alpha,k}; \QQ) \rightarrow H^{\bullet}(F; \QQ)$ on rational 
cohomology.

Let $\alpha \in [k]^r$ with $|\alpha| = n$.  
Theorem~\ref{affine-paving} and Lemma~\ref{set-theoretic-stratification} imply that the cohomology ring 
$H^{\bullet}(X_{\alpha,k}; \ZZ)$ has $\ZZ$-basis given by the classes 
$\{ [\overline{C_{\III} }] \,:\, \text{ $\III$ covers $[k]$ } \}$
of cell closures 
corresponding to covering set sequences $\III = (I_1, \dots, I_r)$ in $[k]$ of type $\alpha$.
In light of Theorem~\ref{main-theorem}, it is natural to pose the following problem.

\begin{problem}
\label{cell-representatives}
Given a covering set sequence $\III = (I_1, \dots, I_r)$ in $[k]$ of type $\alpha$, 
let $C_{\III} \subseteq X_{\alpha,k}$
be the corresponding cell. Find a polynomial in $\ZZ[\xx_n]$ which represents 
$[\overline{C_{\III}}]$ in $H^{\bullet}(X_{\alpha,k};\ZZ) = (R_{\alpha,k})^{\symm_{\alpha}}$.
\end{problem}

When $k = r = n$, the set sequences $\III$
 appearing in Problem~\ref{cell-representatives} may be identified with permutations in $\symm_n$
 and the classical {\em Schubert polynomials} solve Problem~\ref{cell-representatives}.
 When $k \leq r = n$, a solution to Problem~\ref{cell-representatives} is given by the 
 `Fubini word Schubert polynomials' found in \cite{PR}.

Assuming the map $\iota^*$ is {\em surjective},
the Leray-Hirsch Theorem would give  the following isomorphism
of $H^{\bullet}(X_{\alpha,k}; \QQ)$-modules:
\begin{equation}
\label{leray-hirsch-isomorphism}
H^{\bullet}(Y_{\alpha,k}; \QQ) \cong H^{\bullet}(F; \QQ) \otimes_{\QQ} H^{\bullet}(X_{\alpha,k}; \QQ) =
H^{\bullet}(F; \QQ) \otimes (S_{\alpha,k})^{\symm_{\alpha}}.
\end{equation}
The surjectivity of $\iota^*$ would likely follow by constructing an affine paving of the $n$-fold 
product of $k-1$-dimensional complex projective spaces $(\mathbb{P}^{k-1})^n$
whose initial part gives an affine paving of $Y_{\alpha,k} \subseteq (\mathbb{P}^{k-1})^n$.
Such a paving was constructed in the case $\alpha = (m, 1, \dots, 1)$ in \cite{RW}.
By the K\"unneth Theorem, the ring $H^{\bullet}(F; \QQ)$ is isomorphic to a direct product
of classical coinvariant rings, so 
that the isomorphism \eqref{leray-hirsch-isomorphism} would prove 
Conjecture~\ref{y-cohomology-conjecture} at the level of graded $\QQ$-vector spaces.

Assume that the dimension vector $\alpha = (d, \dots, d) \in [k]^r$ is constant.
Corollary~\ref{ungraded-combinatorial-model} describes the cohomology ring 
$H^{\bullet}(X_{\alpha,k}; \QQ) = (S_{\alpha,k})^{\symm_{\alpha}}$ as an ungraded $\QQ[\symm_r]$-module.
This suggests the following problem.

\begin{problem}
\label{graded-frobenius}
Let $\alpha = (d, \dots, d) \in [k]^r$ be a constant dimension vector.
\begin{itemize}
\item
Determine the isomorphism type of $(S_{\alpha,k})^{\symm_{\alpha}}$ as a 
graded $\symm_{r}$-module.
\item
Determine the isomorphism type of $S_{\alpha,k}$ as a graded $\symm_d \wr \symm_r$-module.
\end{itemize}
\end{problem}

In the case $\alpha = (1, \dots, 1)$, the two bullet points of Problem~\ref{graded-frobenius} coincide
and $S_{\alpha,k}$ gives an algebraic and geometric model for the Delta Conjecture 
at $t = 0$ \cite{HRS, PR}.
A solution to Problem~\ref{graded-frobenius} could suggest an extension of the Delta Conjecture itself
to general dimension vectors $\alpha$.

Borel's presentation of the partial flag variety $\mathcal{F \ell}(\alpha)$ in 
Theorem~\ref{borel-theorem} has a generalization to arbitrary Lie type. In particular, if $G$ is a reductive
complex Lie group and if $P$ is a parabolic subgroup of $G$, the homogeneous space $G/P$ 
has rational
cohomology presented in terms of the coordinate ring  of the Cartan subalgebra $\mathfrak{h}$ of $G$.
It would be interesting to extend our results on the more general spaces $X_{\alpha,k}$ to arbitrary Lie type.

\section{Acknowledgements}
\label{Acknowledgements}

The author is grateful to Sara Billey, Brendan Pawlowski, Vic Reiner, and Andy Wilson for many helpful conversations.
The author thanks Fran\c{c}ois Bergeron for asking how to generalize $X_{n,k}$ to higher-dimensional
subspaces.
The author also thanks the anonymous referees for helpful comments which improved the 
exposition of this paper.
The author was partially supported by NSF Grants DMS-1500838 and
DMS-1953781.

\end{document}